\documentclass[5pt]{article}
\usepackage{amsmath}
\usepackage[latin1]{inputenc}
\usepackage{amsmath,amsthm,amssymb}
\usepackage{graphicx}
\usepackage[all,poly,knot]{xy}

\usepackage[all]{xy}

\usepackage{listings} 
\usepackage{tikz-cd} 
\usepackage{amssymb} 
\usepackage{booktabs}
\usepackage{mathtools}
\usepackage{color}

\renewcommand{\theequation}{\thesection.\arabic{equation}}

\newtheorem{lem}{Lemma}[section]
\newtheorem{thm}{Theorem} [section]
\newtheorem{prop}{Proposition} [section]
\newtheorem{exmp}{Example} [section]
\newtheorem{coro}{Corollary}[section]
\newtheorem{rem}{Remark}[section]

\newtheorem{prob}{Problem}[section]

\title{Discrete Fourier Transform Approach to Cyclically Covering Subspaces of $\mathbb{F}^n_q$}

\author{Yangcheng Li$^{1,}$\footnote{E-mail\,$:$ liyc@m.scnu.edu.cn.} \,\, Pingzhi Yuan$^{1,}$\footnote{Corresponding author. E-mail\,$:$ yuanpz@scnu.edu.cn. Supported By NSF of China No.12571003, 12501006; Basic and Applied Basic Research Foundation of Guangdong Province No. 2024A1515010589.}  \\
	{\small\it  $^{1}$School of Mathematical Sciences, South China Normal University,}\\
	{\small\it Guangzhou 510631, Guangdong, P. R. China} \\
}

\date{}
\begin{document}
\baselineskip15pt \maketitle
\renewcommand{\theequation}{\arabic{section}.\arabic{equation}}
\catcode`@=11 \@addtoreset{equation}{section} \catcode`@=12

\begin{abstract}
Let $q$ be a prime power and $n$ a positive integer. A subspace \( U \subseteq \mathbb{F}_q^n \) is called cyclically covering if the union of all its cyclic shifts covers the whole space \( \mathbb{F}_q^n \). Let \( h_q(n) \) denote the maximum possible codimension of such a subspace. When \(\gcd(q,n)=1\), we derive necessary and sufficient conditions for \(h_q(n)=0\) via Discrete Fourier Transforms, and prove this equality is equivalent to the existence of full-weight codewords in cyclic codes of \(\mathbb{F}_q^n\). We also characterize codimension-$k$ cyclically covering subspaces. {Under suitable coprimality conditions on \(m,n\) and on the multiplicative orders of \(q\), we prove that the vanishing of \(h_q(m)\) and \(h_q(n)\) is preserved under their product.}

Based on these results, we give a unified characterization of \(h_q(n)\) in the case where $q$ and $n$ are primes with \(n>q\) and $q$ being a primitive root modulo $n$. Specifically, \(h_2(n) \geq 2\) and \(h_q(n) = 0\) for \(q \neq 2\). We prove that \(h_3(n) \ge 1\) for every prime \(n > 3\) with odd \(\operatorname{ord}_n(3)\). Moreover, for any prime \(q > 3\), the Generalized Riemann Hypothesis implies the existence of infinitely many primes \(n > q\) such that $q$ is not a primitive root modulo $n$ and \(h_q(n) = 0\). We provide algebraic interpretations for the inequalities \(h_q(mn)\ge\max\{h_q(m),h_q(n)\}\) and \(h_q(mn)\ge h_q(m)+h_q(n)\). Using Galois descent, we prove \(h_{q^m}(n)\le h_q(n)\). Furthermore, we generalize a class of constructions that achieve the upper bound \(\lfloor\log_q(n)\rfloor\). Finally, under the Generalized Riemann Hypothesis, we obtain average lower bounds of \(h_q(n)\) for $q=2,3$.
\end{abstract}


\section{Introduction}
A finite field, denoted as $\mathbb{F}_{q}$, where $q$ is a prime power, is a field that contains a finite number of elements. For $n\in\mathbb{N}$, let $\{\boldsymbol{e}_0, \boldsymbol{e}_1, \dots, \boldsymbol{e}_{n-1}\}$ be the standard basis for $\mathbb{F}^n_q$, the indices of vectors in $\mathbb{F}^n_q$ will be taken modulo $n$ (in particular, we set $\boldsymbol{e}_n = \boldsymbol{e}_0$). Define the cyclic shift operator $\tau : \mathbb{F}^n_q\to \mathbb{F}^n_q$ by
\[\tau: \sum_{i=0}^{n-1}a_i\boldsymbol{e}_i \mapsto \sum_{i=0}^{n-1}a_i\boldsymbol{e}_{i+1}.\]
We say that a subspace $U \subset \mathbb{F}^n_q$ is cyclically covering if $\bigcup_{i=0}^{n-1}\tau^i(U) =\mathbb{F}^n_q$. For any $n\in\mathbb{N}$, let $h_q(n)$ denote the largest possible codimension of a cyclically covering subspace of $\mathbb{F}^n_q$.

For general \(n\), determining the value of \(h_q(n)\) is quite difficult. Relatively few values of \(h_q(n)\) have been determined. Motivated by Isbell's conjecture \cite{Isbell56, Isbell60}, Cameron, Ellis, and Raynaud \cite{Cameron-Ellis-Raynaud} in 2019 studied several properties of \(h_q(n)\), determined some of its special values, and established upper and lower bounds for it. Their main results are summarized as follows.
\begin{lem}\cite{Cameron-Ellis-Raynaud} \label{lem1.1} Let \(q\) be a power of prime \(p\), and \(n,m,d,k \in \mathbb{N}\), then the following hold.

$(\mathrm{i})$ $h_2(n) \ge 2$ for all odd integers $n > 3$.

$(\mathrm{ii})$ $h_q(nm) \geq \max\{h_q(n), h_q(m)\}.$

$(\mathrm{iii})$ \(h_q(n) \leq \lfloor\log_q(n)\rfloor\).

$(\mathrm{iv})$ $h_q(q^d - 1) = d - 1 = \lfloor\log_q(q^d - 1)\rfloor.$

$(\mathrm{v})$ \(h_q(M/c) \geq kd + k - c\frac{q^k - 1}{q - 1}\), where \(M = (q - 1)\left(\sum_{r=0}^d q^{kr}\right)\), and \(M\) has a divisor \(c \in \mathbb{N}\) such that \(c < (q - 1)\frac{q^{kd} - q^{-kd}}{q^k - 1}\).

$(\mathrm{vi})$ \(h_q\left(\sum_{r=0}^d q^{kr}\right) = kd\), if \(\gcd(d + 1, q^k - 1) = 1\).

$(\mathrm{vii})$ $h_q(kp^d) = 0$ if \(k \mid q - 1\).

$(\mathrm{viii})$ \(h_2(n) = 0\) if and only if \(n = 2^d\) for some \(d \in \mathbb{N} \cup \{0\}\), and \(h_2(n) = 1\) if and only if \(n = 3\).
\end{lem}

In 1991, Cameron (see \cite{Cameron} Problem 190) posed the following problem in an equivalent form:
\begin{prob}\cite{Cameron}\label{prob1.1}
Does $h_2(n)\to\infty$ as $n\to\infty$ over the odd integers or is $h_2(n) =2$ for infinitely many odd $n$?
\end{prob}
Motivated by Problem \ref{prob1.1}, in 2019, Aaronson, Groenland and Johnston \cite{Aaronson-Groenland-Johnston} investigated the cyclically covering subspaces of \(\mathbb{F}_2^n\). Their main conclusions are summarized as follows.
\begin{lem}\cite{Aaronson-Groenland-Johnston} \label{lem1.2} Let \(q\) be a power of prime \(p\), and \(n,m,\ell \in \mathbb{N}\), then the following hold.

$(\mathrm{i})$ \(h_q(mn) \geq h_q(m) + h_q(n)\).

$(\mathrm{ii})$ \(h_q(pn) \leq ph_q(n)\).

$(\mathrm{iii})$ \(h_q(\ell p^d) = 0\) for any \(\ell < q\).

$(\mathrm{iv})$ \(h_2(t) = 2\) if \(t > 3\) is a prime for which \(2\) is a primitive root.

$(\mathrm{v})$ \(h_q(t) = 0\), where \(q\) is an odd prime, and \(t > q\) is a prime with \(q\) as a primitive root.
\end{lem}
Based on conclusions $(\mathrm{i})$ and $(\mathrm{iv})$ of Lemma \ref{lem1.2}, a positive answer to Problem \ref{prob1.1} can be given provided that Artin's conjecture holds true. Artin conjectured that $2$ is a primitive root modulo infinitely many primes. More generally, for any non-square positive integer \(n\), there are infinitely many primes \(p\) for which \(n\) is a primitive root modulo \(p\). Artin's conjecture follows from the Generalized Riemann Hypothesis (Hooley \cite{Hooley}). Though no \(n\) is known to satisfy it, Heath-Brown \cite{Heath-Brown} proved it holds for at least one of \(\{2, 3, 5\}\).

Furthermore, Aaronson, Groenland and Johnston \cite{Aaronson-Groenland-Johnston} posed several interesting and challenging problems, such as:
\begin{prob}\cite{Aaronson-Groenland-Johnston}\label{prob1.2}
For which \( n \in \mathbb{N} \) is \( h_q(n) = 0 \)?
\end{prob}

In 2024, Huang \cite{Huang} obtained a necessary and sufficient condition for \(h_q(n) = 0\) when \(\gcd(q, n) = 1\).
His main result shows that this problem can be reduced to computing the values of the trace function over finite fields.
He also derived the following conclusions.
\begin{lem}\cite{Huang}\label{lem1.3} Let \( p \) be an odd prime, and let \( q \) be a power of \( p \). For any non-negative integer \( d \), the following statements hold:

$(\mathrm{i})$ \( h_q\left(p^d(q + 1)\right) = 0 \).

$(\mathrm{ii})$ \( h_q\left(2p^d(q - 1)\right) = 0 \) if \(4 \mid q + 1 \).

$(\mathrm{iii})$ Let \( \ell \) be an odd prime number such that \( q \) is a primitive root modulo \( 2\ell \). If \( q \) is relatively prime to \( \ell - 1 \), then \( h_q\left(2p^d\ell\right) = 0 \).
\end{lem}

In 2025, Sun, Ma and Zeng \cite{Sun-Ma-Zeng} investigated the cyclically covering subspaces of the finite field \(\mathbb{F}_{q^n}\) and determined the value of \(h_2(n)\) for certain special values of \(n\). In particular, they showed that \(h_2(21) = 4\). Finally, they established several lower bounds for \(h_q(n)\) in the case where \(\gcd(q, n) = 1\).

In 2025, Li and Yuan \cite{Li-Yuan} proved that the problem of determining \(h_q(n) = 0\) can be reduced to the case where \(\gcd(q, n) = 1\). Specifically, they established the following result.
\begin{lem}\cite{Li-Yuan}\label{lem1.4} 
Let \( q \) be a power of a prime $p$, and let \( n \) be a positive integer satisfying \( \gcd(p, n) = 1 \). Then for any non-negative integer \( k \), we have \( h_q(np^k) = 0 \) if and only if \( h_q(n) = 0 \).
\end{lem}

In 2026, Li, Yuan, Li and Zeng \cite{Li-Yuan-Li-Zeng} obtained several sufficient conditions for \(h_q(n) = 0\), which generalize some existing results. Specifically, they showed that \( h_q(\ell^t) = 0 \) whenever \( q \) is a primitive root modulo \( \ell^t \). Moreover, they proved that if \( n \) is odd and \( h_q(n) = 0 \), then also \( h_q(2n) = 0 \). 

We also mention a problem related to covering vector spaces over \(\mathbb{F}_q\). Luh \cite{Luh} proved that any vector space \(V\) over a finite field \(\mathbb{F}_q\) can be expressed as the union of \(|\mathbb{F}_q| + 1\) proper subspaces, and such a collection of subspaces is unique up to automorphisms of \(V\).

In this paper, we derive necessary and sufficient conditions for \(h_q(n)=0\) with \(\gcd(q,n)=1\) via Discrete Fourier Transforms, and prove this equality is equivalent to the existence of full-weight codewords in cyclic codes of \(\mathbb{F}_q^n\). We also characterize codimension-$k$ cyclically covering subspaces. 
Based on these results, we first give a unified characterization of \(h_q(n)\) in the case where $q$ and $n$ are primes with \(n>q\) and $q$ being a primitive root modulo $n$. Specifically, \(h_2(n) \geq 2\) and \(h_q(n) = 0\) for \(q \neq 2\). 

In existing literature (see \cite{Aaronson-Groenland-Johnston,Cameron-Ellis-Raynaud}), the characterization of \(h_2(n)=0\) is well established. By contrast, only scattered results are available for \(h_q(n)=0\) in the case of odd characteristic.
We prove that \(h_3(n) \ge 1\) for every prime \(n > 3\) with odd \(\operatorname{ord}_n(3)\). Moreover, for any prime \(q > 3\), the Generalized Riemann Hypothesis implies the existence of infinitely many primes \(n > q\) such that $q$ is not a primitive root modulo $n$ and \(h_q(n) = 0\). And we give an explicit example \(h_5(11)=0\), where $5$ is not a primitive root modulo $11$. We provide algebraic interpretations for the inequalities \(h_q(mn)\ge\max\{h_q(m),h_q(n)\}\) and \(h_q(mn)\ge h_q(m)+h_q(n)\). {We further prove a coprime-product criterion: if \(m,n\ge 2\), the integers \(m,n\) are coprime to each other and to \(q\), the orders \(\operatorname{ord}_m(q)\) and \(\operatorname{ord}_n(q)\) are coprime, and \(h_q(m)=h_q(n)=0\), then \(h_q(mn)=0\). We also extend this criterion to pairwise coprime products of finitely many factors.} Using Galois descent, we prove \(h_{q^m}(n)\le h_q(n)\). Moreover, we generalize a class of constructions that achieve the upper bound \(\lfloor\log_q(n)\rfloor\). Finally, under the Generalized Riemann Hypothesis, we obtain average lower bounds of \(h_q(n)\) for $q=2,3$.

This paper is organized as follows. In Section 2, we first present the definition of the Discrete Fourier Transform along with several fundamental properties. Then we derive the necessary and sufficient conditions for \(h_q(n) = 0\). Section 3 is devoted to specific applications of the above results. {Section 4 gives an algebraic interpretation of \(h_q(mn)\ge\max\{h_q(m),h_q(n)\}\). Section 5 develops tensor-product methods both for the superadditivity inequality and for a coprime-product criterion preserving \(h_q(n)=0\).} In Section 6, we apply Galois descent to prove \(h_{q^m}(n)\leq h_q(n)\). In Section 7, we generalize a class of constructions attaining the upper bound \(\lfloor\log_q n\rfloor\). In Section 8, we establish average lower bounds of \(h_q(n)\) under the Generalized Riemann Hypothesis.

\section{Discrete Fourier Transform}
The Discrete Fourier Transform (DFT) is the main tool used in this paper, so we first present its definition and some basic properties. Throughout this paper, we assume 
$\gcd(n,q)=1$ unless otherwise stated. Denote $[n]$ as the set $\{0,1,\dots,n-1\}$. Let $m=\mathrm{ord}_n(q)$. Then there exists a primitive $n$-th root of unity $\omega$ in the extension field $\mathbb{F}_{q^m}$ over $\mathbb{F}_q$ such that $\omega^n=1$ and $\omega^i\neq1$ for $1\le i\le n-1$. For any vector
\[\boldsymbol{x} = (x_0, x_1, \dots, x_{n-1}) \in \mathbb{F}_q^n,\]
its Discrete Fourier Transform (DFT) is defined as
\[\widehat{\boldsymbol{x}}_k=\operatorname{DFT}(\boldsymbol{x})_k = \sum_{j=0}^{n-1} x_j \,\omega^{jk}, \qquad k = 0, 1, \dots, n-1.\]
In matrix form, this reads
\[\begin{pmatrix}
\widehat{x}_0 \\
\widehat{x}_1 \\
\vdots \\
\widehat{x}_{n-1}
\end{pmatrix}
=
\begin{pmatrix}
1 & 1 & 1 & \cdots & 1 \\
1 & \omega & \omega^2 & \cdots & \omega^{n-1} \\
1 & \omega^2 & \omega^4 & \cdots & \omega^{2(n-1)} \\
\vdots & \vdots & \vdots & \ddots & \vdots \\
1 & \omega^{n-1} & \omega^{2(n-1)} & \cdots & \omega^{(n-1)(n-1)}
\end{pmatrix}
\begin{pmatrix}
x_0 \\
x_1 \\
\vdots \\
x_{n-1}
\end{pmatrix}.\]
The matrix above is a Vandermonde matrix. Since \(\omega\) is a primitive \(n\)-th root of unity, its determinant is non-zero, so the DFT is injective. In particular, it defines an \(\mathbb{F}_q\)-linear injective map from \(\mathbb{F}_q^n\) to \(\mathbb{F}_{q^m}^n\). The Inverse Discrete Fourier Transform (IDFT) is given as follows:  
\[x_j = \operatorname{IDFT}(\widehat{\boldsymbol{x}})_j = \frac{1}{n} \sum_{k=0}^{n-1} \widehat{\boldsymbol{x}}_k \,\omega^{-jk}, \qquad j = 0, 1, \dots, n-1,\]
where \(\frac{1}{n}\) denotes the multiplicative inverse of \(n\) in the field \(\mathbb{F}_q\). We first characterize the image of $\mathbb{F}_q^n$ under the Discrete Fourier Transform.

\begin{lem}\label{Lem0}
The image of \(\mathbb{F}_q^n\) under the Discrete Fourier Transform is
$$\operatorname{DFT}(\mathbb{F}_q^n)= \big\{\,\boldsymbol{y} \in \mathbb{F}_{q^m}^n : y_k^q = y_{qk} \text{ for all } k\in [n]\,\big\}.$$
\end{lem}

\begin{proof}
Let $\operatorname{DFT}(\mathbb{F}_q^n)$ denote the image of $\mathbb{F}_q^n$ under the Discrete Fourier Transform. Define the set
$$\mathcal{S}=\big\{\,\boldsymbol{y} \in \mathbb{F}_{q^m}^n : y_k^q = y_{qk} \text{ for all } k\in [n]\,\big\}.$$
For any $\boldsymbol{x}\in\mathbb{F}_q^n$, we have
\[(\widehat{\boldsymbol{x}}_k)^q=\left(\sum_{j=0}^{n-1}x_j\omega^{jk}\right)^q=\sum_{j=0}^{n-1}x_j^q\omega^{qjk}=\sum_{j=0}^{n-1}x_j\omega^{qjk}=\widehat{\boldsymbol{x}}_{qk}.\]
Consequently, $\operatorname{DFT}(\mathbb{F}_q^n)\subset \mathcal{S}$.

Conversely, take an arbitrary $\boldsymbol{y}\in \mathcal{S}$. We compute the $q$-th power of $x_j$ as follows:
\[x_j^q = \left( \frac{1}{n} \sum_{k=0}^{n-1} y_k \omega^{-jk} \right)^q = \left(\frac{1}{n}\right)^q \sum_{k=0}^{n-1} y_k^q (\omega^{-jk})^q 
= \frac{1}{n} \sum_{k=0}^{n-1} y_{qk} \omega^{-jqk} =\frac{1}{n} \sum_{l=0}^{n-1} y_l \omega^{-jl} = x_j,\]
where $l=qk$. Hence $x_j \in \mathbb{F}_q$ for all $j$, which implies $\boldsymbol{x} \in \mathbb{F}_q^n$ and $\boldsymbol{y} = \operatorname{DFT}(\boldsymbol{x})$. Consequently, $\mathcal{S}\subset \operatorname{DFT}(\mathbb{F}_q^n)$.
\end{proof}

The rationale for using the Discrete Fourier Transform to study cyclically covering subspaces of $\mathbb{F}_q^n$ lies in the following two lemmas. Recall the cyclic shift operator $\tau$ on $\mathbb{F}_q^n$:
$$\tau: (x_0,x_1,\dots,x_{n-1})\mapsto (x_{n-1},x_0,x_1,\dots,x_{n-2}).$$

\begin{lem}\label{Lem01}
For any $k = 0,1,\dots,n-1$, we have
\begin{equation}
\widehat{\tau(\boldsymbol{x})}_k=\omega^k\cdot\widehat{\boldsymbol{x}}_k.
\end{equation}
\end{lem}

\begin{proof}
Since $\tau(\boldsymbol{x})_j = x_{j-1}$ for $j = 0,1,\dots,n-1$, where all subscripts are taken modulo $n$, we obtain
\[\widehat{\tau(\boldsymbol{x})}_k= \sum_{j=0}^{n-1} \tau(\boldsymbol{x})_j \omega^{jk}= \sum_{j=0}^{n-1} x_{j-1} \omega^{jk}
= \sum_{i=0}^{n-1} x_i \omega^{(i+1)k}= \omega^k \cdot \sum_{i=0}^{n-1} x_i \omega^{ik}= \omega^k \cdot \widehat{\boldsymbol{x}}_k.\]
\end{proof}

\begin{lem}\label{Lem02}
For any $\boldsymbol{x}=(x_0,\dots,x_{n-1}),\,\boldsymbol{y}=(y_0,\dots,y_{n-1})\in\mathbb{F}_q^n$, let $(\boldsymbol{x},\boldsymbol{y})$ denote the standard inner product, i.e., $(\boldsymbol{x},\boldsymbol{y})=\sum_{i=0}^{n-1}x_iy_i.$ We have
\begin{equation}
(\boldsymbol{x},\boldsymbol{y})=\frac{1}{n}\sum_{k=0}^{n-1}\widehat{\boldsymbol{x}}_k\widehat{\boldsymbol{y}}_{-k}.
\end{equation}
\end{lem}

\begin{proof}
We proceed by direct computation.
\[(\boldsymbol{x},\boldsymbol{y})=\sum_{i=0}^{n-1}x_iy_i=\sum_{i=0}^{n-1}\left(\frac{1}{n}\sum_{k=0}^{n-1}\widehat{\boldsymbol{x}}_k\omega^{-ik}\right)\left(\frac{1}{n}\sum_{l=0}^{n-1}\widehat{\boldsymbol{y}}_l\omega^{-il}\right)=\frac{1}{n^2}\sum_{k=0}^{n-1}\sum_{l=0}^{n-1}\widehat{\boldsymbol{x}}_k\widehat{\boldsymbol{y}}_l\sum_{i=0}^{n-1}\omega^{-i(k+l)}.\]
By the orthogonality relations of the $n$-th roots of unity, we have
\[\sum_{i=0}^{n-1}\omega^{-i(k+l)}=
\begin{cases}
n, & k+l\equiv 0\pmod{n}, \\
0, & \text{otherwise}.
\end{cases}\]
Thus, the inner sum is non-zero if and only if $l\equiv -k\pmod{n}$. Substituting this into the expression yields
\begin{equation*}
(\boldsymbol{x},\boldsymbol{y})=\frac{1}{n}\sum_{k=0}^{n-1}\widehat{\boldsymbol{x}}_k\widehat{\boldsymbol{y}}_{-k}.
\end{equation*}
\end{proof}

Let $\gcd(n,q)=1$. For any integer $0\le j\le n-1$, the cyclotomic coset of $j$ modulo $n$ with respect to $q$ is defined as
$$C_j=\{j,jq,jq^2,\dots,jq^{m_j-1}\}\pmod{n},$$
where $m_j$ denotes the minimal positive integer satisfying $jq^{m_j}\equiv j\pmod{n}$.

Since \(\gcd(n,q)=1\), the polynomial \(x^n-1\) factors into a product of pairwise coprime irreducible polynomials over $\mathbb{F}_q[x]$, i.e.,
\[x^n-1=f_1(x)f_2(x)\cdots f_r(x).\]
Each irreducible factor $f_s(x)$ corresponds uniquely to a cyclotomic coset $C_j$. If $f_s(\omega^j)=0$, then all roots of $f_s(x)$ are precisely $\{\omega^k\mid k\in C_j\}$. We denote this cyclotomic coset by $C(s)$ to indicate that it corresponds to $f_s(x)$.

\begin{lem}\label{Lem03}
We have the following direct sum decomposition.
\[\mathbb{F}_q^n=W_1\oplus\cdots\oplus W_r,\]
where $W_i=\ker f_i(\tau)=\{\boldsymbol{x}\in\mathbb{F}_q^n :  f_i(\tau)(\boldsymbol{x})=\boldsymbol{0}\}$.
\end{lem}

\begin{proof}
We first prove that $\mathbb{F}_q^n = W_1+W_2+\cdots+W_r$. For each integer $1\le i\le r$, define the auxiliary polynomial 
$g_i(x)=\prod_{j\neq i}f_j(x).$
It is immediate that $f_i(x)g_i(x)=x^n-1$. Since all $f_i(x)$ are pairwise coprime, we have $\gcd\big(g_1(x),g_2(x),\dots,g_r(x)\big)=1$. Thus, there exist polynomials $a_1(x),a_2(x),\dots,a_r(x)\in\mathbb{F}_q[x]$ satisfying
$$\sum_{i=1}^r a_i(x)g_i(x)=1.$$
Substituting the cyclic shift operator $\tau$ for the indeterminate $x$ yields the operator identity
$$\sum_{i=1}^r a_i(\tau)g_i(\tau)=I,$$
where $I$ denotes the identity operator on $\mathbb{F}_q^n$. For an arbitrary vector $\boldsymbol{v}\in\mathbb{F}_q^n$, we have
$$\boldsymbol{v}=I(\boldsymbol{v})=\sum_{i=1}^r a_i(\tau)g_i(\tau)(\boldsymbol{v}).$$
Set $\boldsymbol{v}_i=a_i(\tau)g_i(\tau)(\boldsymbol{v})$. A direct calculation gives
\[f_i(\tau)(\boldsymbol{v}_i) = f_i(\tau)a_i(\tau)g_i(\tau)(\boldsymbol{v}) = a_i(\tau)\cdot\big(f_i(\tau)g_i(\tau)\big)(\boldsymbol{v}) = a_i(\tau)\cdot(\tau^n-I)(\boldsymbol{v}) = a_i(\tau)(\boldsymbol{0}) = \boldsymbol{0}.\]
Hence $\boldsymbol{v}_i\in W_i$, which implies every vector $\boldsymbol{v}$ can be written as a sum of elements from $W_1,W_2,\dots,W_r$, i.e., $\mathbb{F}_q^n=W_1+W_2+\cdots+W_r$.

We now verify this sum is direct, namely
$$W_i\cap\left(\sum_{j\neq i}W_j\right)=\{\boldsymbol{0}\},\quad \forall\,1\le i\le r.$$
Take any $\boldsymbol{v}\in W_i\cap\left(\sum_{j\neq i}W_j\right)$. Since $\boldsymbol{v}\in W_i$, we have $f_i(\tau)(\boldsymbol{v})=\boldsymbol{0}$. Meanwhile, $\boldsymbol{v}\in\sum_{j\neq i}W_j$ ensures there exist $\boldsymbol{v}_j\in W_j$ ($j\neq i$) such that $\boldsymbol{v}=\sum_{j\neq i}\boldsymbol{v}_j$, with $f_j(\tau)(\boldsymbol{v}_j)=\boldsymbol{0}$ for all $j\neq i$. Recall $g_i(x)=\prod_{j\neq i}f_j(x)$; then
$$g_i(\tau)(\boldsymbol{v})=g_i(\tau)\left(\sum_{j\neq i}\boldsymbol{v}_j\right)=\sum_{j\neq i}g_i(\tau)(\boldsymbol{v}_j)=\boldsymbol{0}.$$
As $\gcd\big(f_i(x),g_i(x)\big)=1$, there exist polynomials $a(x),b(x)\in\mathbb{F}_q[x]$ such that
$$a(x)f_i(x)+b(x)g_i(x)=1.$$
Applying the same operator substitution:
\[\boldsymbol{v} = I(\boldsymbol{v}) = a(\tau)f_i(\tau)(\boldsymbol{v}) + b(\tau)g_i(\tau)(\boldsymbol{v}) = a(\tau)(\boldsymbol{0}) + b(\tau)(\boldsymbol{0}) = \boldsymbol{0}.\]
Therefore $W_i\cap\big(\sum_{j\neq i}W_j\big)=\{\boldsymbol{0}\}$ holds for all indices $i$, so the decomposition is a direct sum.
\end{proof}

Let $\boldsymbol{x}\in\mathbb{F}_q^n$. If $f_i(\tau)(\boldsymbol{x})=\boldsymbol{0}$, then $f_i(x)$ is called an annihilating polynomial of $\boldsymbol{x}$, and in this case $\boldsymbol{x}\in W_i$. For $\boldsymbol{x}\in\mathbb{F}_q^n$, the support of its Discrete Fourier Transform $\widehat{\boldsymbol{x}}\in\mathbb{F}_{q^m}^n$ is defined as
$$\operatorname{supp}(\widehat{\boldsymbol{x}})=\{k\in[n]:\widehat{\boldsymbol{x}}_k\neq 0\}.$$
The support \(\operatorname{supp}(\widehat{\boldsymbol{x}})\) of \(\boldsymbol{x}\) plays an important role in the subsequent analysis. We first present the following lemma.

\begin{lem}\label{Lem1}
For any $\boldsymbol{x}\in\mathbb{F}_q^n$, the support $\operatorname{supp}(\widehat{\boldsymbol{x}})$ is a union of cyclotomic cosets.
\end{lem}

\begin{proof}
This follows from the equality below:
\[(\widehat{\boldsymbol{x}}_k)^q=\left(\sum_{j=0}^{n-1}x_j\omega^{jk}\right)^q=\sum_{j=0}^{n-1}x_j^q\omega^{qjk}=\sum_{j=0}^{n-1}x_j\omega^{qjk}=\widehat{\boldsymbol{x}}_{qk}.\]
\end{proof}

\begin{lem}\label{Lem2}
For any \(1\leq i\leq r\), if $\boldsymbol{x}\in W_i$ and $\boldsymbol{x}\neq \boldsymbol{0}$, then $\operatorname{supp}(\widehat{\boldsymbol{x}})=C(i)$. Furthermore, suppose $\boldsymbol{x}\in W_{i_1}\oplus\cdots\oplus W_{i_s}$ with $\boldsymbol{x}=\boldsymbol{\alpha}_1+\cdots+\boldsymbol{\alpha}_s$ and $\boldsymbol{\alpha}_t\neq \boldsymbol{0}$ for $1\leq t\leq s$. Then $\operatorname{supp}(\widehat{\boldsymbol{x}})=\bigcup_{t=1}^s C(i_t)$, where $1\leq i_1,\dots,i_s\leq r$.
\end{lem}

\begin{proof}
Since $\boldsymbol{x}\in W_i$, we have $f_i(\tau)(\boldsymbol{x})=\boldsymbol{0}$.
Given $\widehat{\tau^i(\boldsymbol{x})}_k=\omega^{ik}\widehat{\boldsymbol{x}}_k$, it yields
$$\widehat{f_i(\tau)(\boldsymbol{x})}_k=f_i(\omega^k)\cdot\widehat{\boldsymbol{x}}_k=0.$$
Hence, $\widehat{\boldsymbol{x}}_k$ can be non-zero only if $f_i(\omega^k)=0$. Consequently, $\operatorname{supp}(\widehat{\boldsymbol{x}})\subset C(i)$. Since \(\boldsymbol{x}\neq \boldsymbol{0}\), we have \(\operatorname{supp}(\widehat{\boldsymbol{x}})\neq\emptyset\). Combining Lemma \ref{Lem1}, we have $\operatorname{supp}(\widehat{\boldsymbol{x}})=C(i)$.

Suppose $\boldsymbol{x}\in W_{i_1}\oplus\cdots\oplus W_{i_s}$, then $\boldsymbol{x}=\boldsymbol{\alpha}_1+\cdots+\boldsymbol{\alpha}_s$, where each $\boldsymbol{\alpha}_t\neq\boldsymbol{0}$ and $\boldsymbol{\alpha}_t\in W_{i_t},1\leq t\leq s$. We have $\widehat{\boldsymbol{x}}=\widehat{\boldsymbol{\alpha}}_1+\cdots+\widehat{\boldsymbol{\alpha}}_s$, and $\operatorname{supp}(\widehat{\boldsymbol{\alpha}}_t)\cap\operatorname{supp}(\widehat{\boldsymbol{\alpha}}_k)=\emptyset$ for $t\neq k$.
Hence
\[\operatorname{supp}(\widehat{\boldsymbol{x}})=\bigcup_{t=1}^s\operatorname{supp}(\widehat{\boldsymbol{\alpha}}_t)\subset \bigcup_{t=1}^s C(i_t).\]
Combining Lemma \ref{Lem1}, we have $\operatorname{supp}(\widehat{\boldsymbol{x}})=\bigcup_{t=1}^s C(i_t)$, $1\leq i_1,...,i_s\leq r$.
\end{proof}

Let $U$ be a subspace of $\mathbb{F}_q^n$. Define $\widehat{U}=\{\widehat{\boldsymbol{u}}: \boldsymbol{u}\in U\}$.
The support of $U$ is given by
$$\operatorname{supp}(\widehat{U})=\{j\in[n]: \exists\,\boldsymbol{u}\in U,\,\widehat{\boldsymbol{u}}_j\neq 0\}=\bigcup_{\boldsymbol{u}\in U}\operatorname{supp}(\boldsymbol{\widehat{u}}).$$

\begin{lem}\label{Lem3}
For any $1\le i\le n-1$, we have $\operatorname{supp}(\widehat{\tau^i(U)})=\operatorname{supp}(\widehat{U})$.
\end{lem}

\begin{proof}
Since $\widehat{\tau^i(\boldsymbol{u})}_k=\omega^{ik}\widehat{\boldsymbol{u}}_k$ and $\omega^{ik}\neq 0$, we get $\widehat{\tau^i(\boldsymbol{u})}_k\neq 0 \iff \widehat{\boldsymbol{u}}_k\neq 0$. Thus $\operatorname{supp}(\widehat{\tau^i(U)})=\operatorname{supp}(\widehat{U})$.
\end{proof}

\begin{lem}\label{Lem4}
Let \(f(x)\) be an irreducible divisor of \(x^n-1\), set \(\deg f(x)=d\) and \(W=\ker f(\tau)\). Then for any \(\boldsymbol{\alpha}\in W\setminus\{\boldsymbol 0\}\), the set \(\{\boldsymbol{\alpha},\tau(\boldsymbol{\alpha}),\dots,\tau^{d-1}(\boldsymbol{\alpha})\}\) forms a basis of $W$. Moreover, for all \(\boldsymbol{\beta}\in W\setminus\{\boldsymbol 0\}\), we have
\(\operatorname{supp}(\widehat{\boldsymbol{\beta}})=\operatorname{supp}(\widehat{\boldsymbol{\alpha}})=\operatorname{supp}(\widehat W).\)
\end{lem}

\begin{proof}
For any \(\boldsymbol{\alpha}\in W\setminus\{\boldsymbol 0\}\), let $m_{\boldsymbol{\alpha}}(x)\in\mathbb{F}_q[x]$ denote the minimal polynomial of $\boldsymbol{\alpha}$ satisfying $m_{\boldsymbol{\alpha}} (\tau)(\boldsymbol{\alpha})=\boldsymbol{0}$.
Since $f(\tau)(\boldsymbol{\alpha})=\boldsymbol{0}$, write the polynomial division $f(x)=q(x)m_{\boldsymbol{\alpha}}(x)+r(x)$ with $\deg r(x)<\deg m_{\boldsymbol{\alpha}}(x)$.
Substitute the operator $\tau$ and act on $\alpha$:
$$
\boldsymbol{0}=f(\tau)(\boldsymbol{\alpha})=q(\tau)m_{\boldsymbol{\alpha}}(\tau)(\boldsymbol{\alpha})+r(\tau)(\boldsymbol{\alpha})=r(\tau)(\boldsymbol{\alpha}).
$$
By minimality of $m_{\boldsymbol{\alpha}}(x)$, we deduce $r(x)=0$, hence $m_{\boldsymbol{\alpha}}(x)\mid f(x)$. As $f(x)$ is irreducible, $m_{\boldsymbol{\alpha}}(x)=f(x)$.

Suppose there exist $a_0,a_1,\dots,a_{d-1}\in\mathbb{F}_q$ such that
$$a_0\boldsymbol{\alpha}+a_1\tau(\boldsymbol{\alpha})+\cdots+a_{d-1}\tau^{d-1}(\boldsymbol{\alpha})=\boldsymbol{0}.$$
Define $g(x)=a_0+a_1x+\cdots+a_{d-1}x^{d-1}$, so $\deg g(x)\le d-1$ and $g(\tau)(\boldsymbol{\alpha})=\boldsymbol{0}$.
From $m_{\boldsymbol{\alpha}}(x)=f(x)$ with $\deg f(x)=d$, we get $g(x)=0$, which yields linear independence of $\{\boldsymbol{\alpha},\tau(\boldsymbol{\alpha}),\dots,\tau^{d-1}(\boldsymbol{\alpha})\}$.
Since $\dim W=d$, this set forms a basis of $W$.
The above argument holds for all nonzero $\boldsymbol{\alpha}\in W$.

Now take arbitrary $\boldsymbol{\beta}\in W$, then
$$\boldsymbol{\beta}=k_0\boldsymbol{\alpha}+k_1\tau(\boldsymbol{\alpha})+\dots+k_{d-1}\tau^{d-1}(\boldsymbol{\alpha})=g(\tau)(\boldsymbol{\alpha}),$$
where $g(x)=k_0+k_1x+\dots+k_{d-1}x^{d-1}\in\mathbb{F}_q[x]$. We have $\widehat{\boldsymbol{\beta}}_k=\widehat{g(\tau)(\boldsymbol{\alpha})}_k=g(\omega^k)\widehat{\boldsymbol{\alpha}}_k$, hence $\operatorname{supp}(\widehat{\boldsymbol{\beta}})\subseteq\operatorname{supp}(\widehat{\boldsymbol{\alpha}})$. If $\boldsymbol{\beta}\neq\boldsymbol{0}$, then $\{\boldsymbol{\beta},\tau(\boldsymbol{\beta}),\dots,\tau^{d-1}(\boldsymbol{\beta})\}$ is also a basis for $W$. Repeating the above reasoning gives $\operatorname{supp}(\widehat{\boldsymbol{\alpha}})\subseteq\operatorname{supp}(\widehat{\boldsymbol{\beta}})$. Consequently, $\operatorname{supp}(\widehat{\boldsymbol{\alpha}})=\operatorname{supp}(\widehat{\boldsymbol{\beta}})$ for all nonzero $\boldsymbol{\alpha},\boldsymbol{\beta}\in W\setminus\{\boldsymbol 0\}$. It is now immediate that \(\operatorname{supp}(\widehat{W})=\operatorname{supp}(\widehat{\boldsymbol{\alpha}})\).
\end{proof}

As our main concern is whether \(h_q(n)=0\), we restrict our attention to subspaces of dimension \(n-1\). Assume that \(V_{\boldsymbol{\alpha}}\) is an \((n-1)\) dimensional subspace of \(\mathbb{F}_{q}^n\), i.e.,
\[V_{\boldsymbol{\alpha}} = \{\boldsymbol{x} \in \mathbb{F}_q^n : (\boldsymbol{x},\boldsymbol{\alpha}) = 0\}~\text{for some}~\boldsymbol{\alpha} \in \mathbb{F}_q^n\setminus \{\boldsymbol{0}\}.\]

It is worth mentioning that Aaronson, Groenland and Johnston \cite{Aaronson-Groenland-Johnston} introduced the notion of vectors that {\it work together} in 2019. Precisely, a vector \(\boldsymbol{\alpha} \in \mathbb{F}_q^n\) is said to {\it work} if for every vector \(\boldsymbol{x} \in \mathbb{F}_q^n\), there exists an integer \(i\in [n]\) such that the inner product \[(\boldsymbol{\alpha}, \tau^i (\boldsymbol{x})) = 0.\]
This condition is equivalent to \(V_{\boldsymbol{\alpha}}\) being a cyclically covering subspace.

A collection of vectors \(\boldsymbol{\alpha}_1,\boldsymbol{\alpha}_2,\dots,\boldsymbol{\alpha}_k\) is said to {\it work together} if for every vector \(\boldsymbol{x} \in \mathbb{F}_q^n\), there exists an integer \(i \in [n]\) such that
\[(\boldsymbol{\alpha}_1 , \tau^i (\boldsymbol{x})) = (\boldsymbol{\alpha}_2 , \tau^i (\boldsymbol{x})) = \dots = (\boldsymbol{\alpha}_k , \tau^i (\boldsymbol{x})) = 0.\]
Similarly, this characterization is equivalent to the intersection \(\bigcap_{j=1}^k V_{\boldsymbol{\alpha}_j}\) being a cyclically covering subspace.

We now employ the Discrete Fourier Transform to derive a necessary and sufficient condition for \(V_{\boldsymbol{\alpha}}\) to be a cyclically covering subspace of \(\mathbb{F}_q^n\). We first establish two lemmas.

\begin{lem}\label{Lem4-1}
Let \(C_j = \{j\}\) denote the singleton cyclotomic coset of $j$ modulo $n$ with respect to $q$. If $-j\in \operatorname{supp}(\widehat{\boldsymbol{\alpha}})$, that is, $j\in -\operatorname{supp}(\hat{\boldsymbol{\alpha}})$, then \(V_{\boldsymbol{\alpha}}\) is not a cyclically covering subspace of \(\mathbb{F}_q^n\).
\end{lem}

\begin{proof}
To prove that \(V_{\boldsymbol{\alpha}}\) is not a cyclically covering subspace of \(\mathbb{F}_q^n\), it suffices to show that there exists an element \(\boldsymbol{x}\in\mathbb{F}_q^n\) such that \(\bigl(\boldsymbol{x},\tau^i(\boldsymbol{\alpha})\bigr)\neq 0\) for all \(0\le i\le n-1\). We have
\[\bigl(\boldsymbol{x},\tau^i(\boldsymbol{\alpha})\bigr)
=\frac{1}{n}\sum_{k=0}^{n-1}\widehat{\boldsymbol{x}}_k \,\widehat{\tau^i(\boldsymbol{\alpha})}_{-k}
=\frac{1}{n}\sum_{k=0}^{n-1}\widehat{\boldsymbol{x}}_k \widehat{\boldsymbol{\alpha}}_{-k}\omega^{-ik}.\]
Since \(-j \in \operatorname{supp}(\widehat{\boldsymbol{\alpha}})\), we have \(\widehat{\boldsymbol{\alpha}}_{-j} \neq 0\). By Lemma \ref{Lem0}, we have $(\widehat{\boldsymbol{\alpha}}_{-j})^q = \widehat{\boldsymbol{\alpha}}_{q(-j)}= \widehat{\boldsymbol{\alpha}}_{-qj}.$
Since $C_j = \{j\}$ is a singleton cyclotomic coset, we have $qj \equiv j \pmod{n}$. This immediately yields $q(-j) \equiv -j \pmod{n}$. Consequently,
\[(\widehat{\boldsymbol{\alpha}}_{-j})^q = \widehat{\boldsymbol{\alpha}}_{-qj} = \widehat{\boldsymbol{\alpha}}_{-j},\]
which implies $\widehat{\boldsymbol{\alpha}}_{-j} \in \mathbb{F}_q$. Set \(\widehat{\boldsymbol{x}}_j = (\widehat{\boldsymbol{\alpha}}_{-j})^{-1}\) and \(\widehat{\boldsymbol{x}}_k = 0\) for all \(k\neq j\). As \(C_j=\{j\}\), we get \(qj\equiv j\pmod{n}\), which implies \(\widehat{\boldsymbol{x}}_{qj}=\widehat{\boldsymbol{x}}_j\). Thus \((\widehat{\boldsymbol{x}}_k)^q = \widehat{\boldsymbol{x}}_{qk}\) holds for all \(k\in [n]\).
It follows that \(\widehat{\boldsymbol{x}}\) uniquely corresponds to some \(\boldsymbol{x}\in\mathbb{F}_q^n\), namely \(\boldsymbol{x} = \operatorname{IDFT}(\widehat{\boldsymbol{x}})\in\mathbb{F}_q^n\). Substituting yields
\[\bigl(\boldsymbol{x},\tau^i(\boldsymbol{\alpha})\bigr)
=\frac{1}{n}\sum_{k=0}^{n-1}\widehat{\boldsymbol{x}}_k \widehat{\boldsymbol{\alpha}}_{-k}\omega^{-ik}
=\frac{1}{n}\widehat{\boldsymbol{x}}_j\widehat{\boldsymbol{\alpha}}_{-j} \omega^{-ij}
=\frac{1}{n}\omega^{-ij}\neq 0.\]
Therefore, \(V_{\boldsymbol{\alpha}}\) is not a cyclically covering subspace of \(\mathbb{F}_q^n\).
\end{proof}

\begin{coro}\label{Lem5}
If \(0 \in \operatorname{supp}(\hat{\boldsymbol{\alpha}})\), then \(V_{\boldsymbol{\alpha}}\) is not a cyclically covering subspace of $\mathbb{F}^n_q$.
\end{coro}

\begin{proof}
For any positive integer \(n\) and prime power \(q\), \(C_0 = \{0\}\) is a singleton cyclotomic coset. The desired result follows from Lemma \ref{Lem4-1}.
\end{proof}
By Corollary \ref{Lem5}, we may assume that \(0\notin\operatorname{supp}(\widehat{\boldsymbol{\alpha}})\). Since \(\operatorname{supp}(\widehat{\boldsymbol{\alpha}})\) is a union of cyclotomic cosets, \(-\operatorname{supp}(\widehat{\boldsymbol{\alpha}})\) is also a union of cyclotomic cosets.

\begin{lem}\label{Lem6-1}
Let \( S = -\operatorname{supp}(\hat{\boldsymbol{\alpha}}) = \bigcup_{j=1}^{s} C(i_j) \) with \(0\notin\operatorname{supp}(\widehat{\boldsymbol{\alpha}})\). Then \(V_{\boldsymbol{\alpha}}\) is not a cyclically covering subspace of \(\mathbb{F}_q^n\) if and only if there exists \(\widehat{\boldsymbol{c}}\in\operatorname{DFT}(\mathbb{F}_q^n)\) with \(\operatorname{supp}(\widehat{\boldsymbol{c}}) \subseteq -\operatorname{supp}(\widehat{\boldsymbol{\alpha}})\) such that
\[\operatorname{wt}(\boldsymbol{c})=\operatorname{wt}(\operatorname{IDFT}(\widehat{\boldsymbol{c}}))=n,\]
where $\boldsymbol{c}=\operatorname{IDFT}(\widehat{\boldsymbol{c}})$ and \(\operatorname{wt}(\boldsymbol{c})=|\operatorname{supp}(\boldsymbol{c})|\).
\end{lem}

\begin{proof}
If \(V_{\boldsymbol{\alpha}}\) is not a cyclically covering subspace of \(\mathbb{F}_q^n\), then there exists some \(\boldsymbol{x} \in \mathbb{F}_q^n\) such that \(\left(\boldsymbol{x}, \tau^i(\boldsymbol{\alpha})\right) \neq 0\) for all \(0 \le i \le n-1\).  
Define \(\widehat{\boldsymbol{c}}_k = \widehat{\boldsymbol{x}}_k \widehat{\boldsymbol{\alpha}}_{-k}\in\mathbb{F}_{q^m}\). Since \(\boldsymbol{x}, \boldsymbol{\alpha} \in \mathbb{F}_q^n\), we have \((\widehat{\boldsymbol{c}}_k)^q = \widehat{\boldsymbol{c}}_{qk}\) for all \(k\in[n]\), so \(\widehat{\boldsymbol{c}} \in \operatorname{DFT}(\mathbb{F}_q^n)\) and  \(\operatorname{supp}(\boldsymbol{c}) \subseteq -\operatorname{supp}(\widehat{\boldsymbol{\alpha}})\).  
For all \(0\le i\le n-1\), we have
\[0\neq \bigl(\boldsymbol{x},\tau^i(\boldsymbol{\alpha})\bigr)
=\frac{1}{n}\sum_{k=0}^{n-1}\widehat{\boldsymbol{x}}_k \,\widehat{\tau^i(\boldsymbol{\alpha})}_{-k}
=\frac{1}{n}\sum_{k=0}^{n-1}\widehat{\boldsymbol{x}}_k \hat{\boldsymbol{\alpha}}_{-k}\omega^{-ik}= \frac{1}{n} \sum_{k=0}^{n-1} \widehat{\boldsymbol{c}}_k \omega^{-ik}=\operatorname{IDFT}(\widehat{\boldsymbol{c}})_i.\]
This yields that there exists \(\widehat{\boldsymbol{c}}\in\operatorname{DFT}(\mathbb{F}_q^n)\) satisfying \(\operatorname{supp}(\widehat{\boldsymbol{c}}) \subseteq -\operatorname{supp}(\widehat{\boldsymbol{\alpha}})\) such that
\[\operatorname{wt}(\boldsymbol{c})=\operatorname{wt}(\operatorname{IDFT}(\widehat{\boldsymbol{c}}))=n,\]
where $\boldsymbol{c}=\operatorname{IDFT}(\widehat{\boldsymbol{c}})$ and \(\operatorname{wt}(\boldsymbol{c})=|\operatorname{supp}(\boldsymbol{c})|\).

If there exists \(\widehat{\boldsymbol{c}}\in\operatorname{DFT}(\mathbb{F}_q^n)\) with \(\operatorname{supp}(\widehat{\boldsymbol{c}}) \subseteq -\operatorname{supp}(\widehat{\boldsymbol{\alpha}})\) such that \(\operatorname{wt}(\operatorname{IDFT}(\widehat{\boldsymbol{c}}))=n\), we construct \(\widehat{\boldsymbol{x}}\) as follows: set
\[\widehat{\boldsymbol{x}}_k=
\begin{cases}
\widehat{\boldsymbol{c}}_k (\widehat{\boldsymbol{\alpha}}_{-k})^{-1}, & -k\in\operatorname{supp}(\widehat{\boldsymbol{\alpha}}),\\
0, & -k\notin\operatorname{supp}(\widehat{\boldsymbol{\alpha}}).
\end{cases}\]
Then \(\widehat{\boldsymbol{c}}_k = \widehat{\boldsymbol{x}}_k \widehat{\boldsymbol{\alpha}}_{-k}\) holds for every \(k\in[n]\). It is easy to verify that \((\widehat{\boldsymbol{x}}_k)^q=\widehat{\boldsymbol{x}}_{qk}\) for all \(k\in[n]\), which implies \(\widehat{\boldsymbol{x}}\in\operatorname{DFT}(\mathbb{F}_q^n)\). By the Inverse Discrete Fourier Transform, there exists a unique \(\boldsymbol{x}\in\mathbb{F}_q^n\) such that \(\operatorname{DFT}(\boldsymbol{x})=\widehat{\boldsymbol{x}}\). For all \(0\le i\le n-1\), we have
\[\bigl(\boldsymbol{x},\tau^i(\boldsymbol{\alpha})\bigr)
=\frac{1}{n}\sum_{k=0}^{n-1}\widehat{\boldsymbol{x}}_k \,\widehat{\tau^i(\boldsymbol{\alpha})}_{-k}
=\frac{1}{n}\sum_{k=0}^{n-1}\widehat{\boldsymbol{x}}_k \widehat{\boldsymbol{\alpha}}_{-k}\omega^{-ik}=\frac{1}{n}\sum_{k=0}^{n-1}\widehat{\boldsymbol{c}}_k\omega^{-ik}=\operatorname{IDFT}(\widehat{\boldsymbol{c}})_i\neq0.\]
Therefore, there exists \(\boldsymbol{x}\in\mathbb{F}_q^n\) satisfying \(\bigl(\boldsymbol{x},\tau^i(\boldsymbol{\alpha})\bigr)\neq 0\) for every \(0\le i\le n-1\), which implies that \(V_{\boldsymbol{\alpha}}\) is not a cyclically covering subspace of \(\mathbb{F}_q^n\).
\end{proof}

For the convenience of calculation, we restate Lemma \ref{Lem6-1} in the explicit form below.

\begin{lem}\label{Lem6}
Let \( S = -\operatorname{supp}(\widehat{\boldsymbol{\alpha}}) = \bigcup_{j=1}^{s} C(i_j) \) with \(0\notin\operatorname{supp}(\widehat{\boldsymbol{\alpha}})\). Let \( m_j \) denote the cardinality of \( C(i_j) \), and let $r_j$ be the minimal element of $C(i_j)$. Denote \(\mathcal{R}=\{r_1,...,r_s\}\). Take a vector \(\widehat{\boldsymbol{c}}\in\operatorname{DFT}(\mathbb{F}_q^n)\) satisfying \(\operatorname{supp}(\widehat{\boldsymbol{c}}) \subseteq -\operatorname{supp}(\widehat{\boldsymbol{\alpha}})\), whence \((\widehat{\boldsymbol{c}}_k)^q = \widehat{\boldsymbol{c}}_{qk}\). Define
\[P_{\widehat{\boldsymbol{c}}}(t)= \widehat{\boldsymbol{c}}_1 t + \widehat{\boldsymbol{c}}_2 t^2 + \cdots + \widehat{\boldsymbol{c}}_{n-1} t^{n-1} =\sum_{j\in\mathcal{R}}\operatorname{Tr}_{\mathbb{F}_q}^{\mathbb{F}_{q^{m_j}}}(\widehat{\boldsymbol{c}}_jt^j)\in \mathbb{F}_{q^m}[t],\]
where the sum \(\displaystyle\sum_{j\in\mathcal{R}}\) runs over all minimal representatives in \(\bigcup_{t=1}^s C(i_t)\). Then \(V_{\boldsymbol{\alpha}}\) is not a cyclically covering subspace of \(\mathbb{F}_q^n\) if and only if there exists \(\widehat{\boldsymbol{c}}\in\operatorname{DFT}(\mathbb{F}_q^n)\) with \(\operatorname{supp}(\widehat{\boldsymbol{c}}) \subseteq -\operatorname{supp}(\widehat{\boldsymbol{\alpha}})\) such that \(P_{\widehat{\boldsymbol{c}}}(\omega^i) \neq 0\) for all \(0 \le i \le n-1\).
\end{lem}

\begin{proof}
Let \(\widehat{\boldsymbol{c}}\in\operatorname{DFT}(\mathbb{F}_q^n)\) satisfy \(\operatorname{supp}(\widehat{\boldsymbol{c}}) \subseteq -\operatorname{supp}(\widehat{\boldsymbol{\alpha}})\). Define the polynomial \(P(t) = \sum\limits_{k=0}^{n-1} \widehat{\boldsymbol{c}}_k t^k\). As \(0\notin\operatorname{supp}(\widehat{\boldsymbol{\alpha}})\), it follows that \(\widehat{\boldsymbol{c}}_0 = 0\). With the identity \((\widehat{\boldsymbol{c}}_k)^q = \widehat{\boldsymbol{c}}_{qk}\), \(P(t)\) can be rewritten as
\[\begin{aligned}
P(t) &= \widehat{\boldsymbol{c}}_1 t + \widehat{\boldsymbol{c}}_2 t^2 + \cdots + \widehat{\boldsymbol{c}}_{n-1} t^{n-1} \\
&=\sum_{j\in\mathcal{R}} \left( \widehat{\boldsymbol{c}}_j t^j + \widehat{\boldsymbol{c}}_{jq} t^{jq} + \widehat{\boldsymbol{c}}_{jq^2} t^{jq^2} + \cdots + \widehat{\boldsymbol{c}}_{jq^{m_j-1}} t^{jq^{m_j-1}} \right)\\
&=\sum_{j\in\mathcal{R}} \left( \widehat{\boldsymbol{c}}_j t^j + (\widehat{\boldsymbol{c}}_j t^j)^q + (\widehat{\boldsymbol{c}}_j t^j)^{q^2} + \cdots + (\widehat{\boldsymbol{c}}_j t^j)^{q^{m_j-1}} \right)\\
&=\sum_{j\in\mathcal{R}}\operatorname{Tr}_{\mathbb{F}_q}^{\mathbb{F}_{q^{m_j}}}(\widehat{\boldsymbol{c}}_jt^j)=P_{\widehat{\boldsymbol{c}}}(t),
\end{aligned}\]
where the sum \(\displaystyle\sum_{j\in\mathcal{R}}\) runs over all minimal representatives in \(\bigcup_{t=1}^s C(i_t)\).
By Lemma \ref{Lem6-1}, \(\operatorname{wt}(\boldsymbol{c})=\operatorname{wt}(\operatorname{IDFT}(\widehat{\boldsymbol{c}}))=n\) is equivalent to \(P_{\widehat{\boldsymbol{c}}}(\omega^i) \neq 0\) for all \(0 \le i \le n-1\).
\end{proof}

A $k$-dimensional linear subspace \(\mathcal{C}\) of \(\mathbb{F}_q^n\) is called an \([n,k]_q\) linear code. A codeword \(\boldsymbol{c}\in\mathcal{C}\) is said to be a full-weight codeword if \(\operatorname{wt}(\boldsymbol{c})=n\). By Lemma \ref{Lem2}, we reformulate Lemma \ref{Lem6-1} into a problem of determining whether there exists a full-weight codeword in the cyclic code.

\begin{coro}\label{coro2-2}
Let \(S = -\operatorname{supp}(\widehat{\boldsymbol{\alpha}}) = \bigcup_{j=1}^{s} C(i_j)\), and assume \(0\notin\operatorname{supp}(\widehat{\boldsymbol{\alpha}})\). The union of cyclotomic cosets \(\bigcup_{j=1}^{s} C(i_j)\) corresponds to the direct sum of subspaces \(W_{i_1}\oplus \dots \oplus W_{i_s}\) of \(\mathbb{F}_q^n\). Then \(V_{\boldsymbol{\alpha}}\) is not a cyclically covering subspace of \(\mathbb{F}_q^n\) if and only if there exists a full-weight codeword in \(W_{i_1}\oplus \dots \oplus W_{i_s}\).
\end{coro}

\begin{proof}
It follows directly from Lemma \ref{Lem2} and Lemma \ref{Lem6-1}.
\end{proof}

We now present the necessary and sufficient conditions for \(h_q(n) = 0\) when \(\gcd(q,n)=1\).
\begin{thm}\label{Thm1}
Let \(\boldsymbol{\alpha} \in \mathbb{F}_q^n\). Then \(V_{\boldsymbol{\alpha}} = \{\boldsymbol{x} \in \mathbb{F}_q^n : (\boldsymbol{x}, \boldsymbol{\alpha}) = 0\}\) is a cyclically covering subspace of \(\mathbb{F}_{q}^n\) if and only if both of the following conditions hold:

(1) \(0 \notin \operatorname{supp}(\widehat{\boldsymbol{\alpha}})\);

(2) For every vector \(\widehat{\boldsymbol{c}} \in \operatorname{DFT}(\mathbb{F}_q^n)\) with \(\operatorname{supp}(\widehat{\boldsymbol{c}}) \subseteq -\operatorname{supp}(\widehat{\boldsymbol{\alpha}})\), it holds that \(\operatorname{wt}(\boldsymbol{c}) = \operatorname{wt}(\operatorname{IDFT}(\widehat{\boldsymbol{c}})) < n\). This is equivalent to the existence of some \(i \in [n]\) with \(P_{\widehat{\boldsymbol{c}}}(\omega^i) = 0\), and \(P_{\widehat{\boldsymbol{c}}}(t)\) is defined in Lemma \ref{Lem6}.
\end{thm}

\begin{proof} 
This follows from Corollary \ref{Lem5} and Lemma \ref{Lem6}.
\end{proof}

Since every cyclic code over \(\mathbb{F}_q^n\) corresponds to a subspace of the form \(W_{i_1}\oplus \dots \oplus W_{i_s}\), we obtain the following result.

\begin{lem}\label{lem:vanishing-full-weight}
Assume that \(\gcd(n,q)=1\). The following conditions are equivalent:

\(\mathrm{(i)}\) \(h_q(n)=0\);

\(\mathrm{(ii)}\) every cyclic code whose DFT support is a nonempty union of \(q\)-cyclotomic cosets modulo \(n\) that does not contain \(C_0\) contains a full-weight codeword;

\(\mathrm{(iii)}\) every minimal cyclic code corresponding to a nonzero \(q\)-cyclotomic coset modulo \(n\) contains a full-weight codeword.
\end{lem}

\begin{proof}
By Corollary \ref{coro2-2} and Theorem \ref{Thm1}, conditions
\(\mathrm{(i)}\) and \(\mathrm{(ii)}\) are equivalent. The requirement that
the union be nonempty simply excludes the zero code.

Clearly, \(\mathrm{(ii)}\) implies \(\mathrm{(iii)}\). Conversely, let
\(\mathcal C\) be a cyclic code as in \(\mathrm{(ii)}\). By Lemma
\ref{Lem03}, \(\mathcal C\) is a direct sum of one or more minimal cyclic
codes. Under \(\mathrm{(iii)}\), any one of these minimal subcodes contains
a full-weight codeword, which also belongs to \(\mathcal C\). Hence
\(\mathrm{(iii)}\) implies \(\mathrm{(ii)}\).
\end{proof}

\begin{coro}\label{cor:prime-multiplier-equivalence}
Let \(n\) be prime and \(\gcd(n,q)=1\). Any two minimal cyclic codes associated with nonzero \(q\)-cyclotomic cosets modulo \(n\) are multiplier-equivalent, and hence have the same weight distribution.
\end{coro}

\begin{proof}
Let \(C_i\) and \(C_j\) be nonzero \(q\)-cyclotomic cosets modulo \(n\).
Since \(n\) is prime, \(i\) and \(j\) are invertible modulo \(n\). Taking
\[
u\equiv ji^{-1}\pmod n,
\]
we have \(uC_i=C_j\).

Define the multiplier permutation
\[
(\mu_u(\boldsymbol{x}))_r=x_{ur},
\qquad r\in\mathbb Z/n\mathbb Z.
\]
For every \(k\in\mathbb Z/n\mathbb Z\), a change of variables
\(s=ur\) gives
\[
\widehat{\mu_u(\boldsymbol{x})}_k
=\sum_{r=0}^{n-1}x_{ur}\omega^{rk}
=\sum_{s=0}^{n-1}x_s\omega^{s u^{-1}k}
=\widehat{\boldsymbol{x}}_{u^{-1}k}.
\]
Therefore,
\[
\operatorname{supp}\bigl(\widehat{\mu_u(\boldsymbol{x})}\bigr)
=u\,\operatorname{supp}(\widehat{\boldsymbol{x}}).
\]
Hence \(\mu_u\) maps the minimal cyclic code associated with \(C_i\)
onto the one associated with \(C_j\). Since \(\mu_u\) is a coordinate
permutation, it preserves Hamming weight, and the two codes have the
same weight distribution.
\end{proof}

From Theorem \ref{Thm1}, we derive the necessary and sufficient conditions for a $k$ dimensional subspace to be cyclically covering.
\begin{thm}\label{Thm2}
Let \(U=\bigcap_{j=1}^k V_{\boldsymbol{\alpha}_j}\) be a subspace of codimension $k$, where \(\boldsymbol{\alpha}_1,\boldsymbol{\alpha}_2,\dots,\boldsymbol{\alpha}_k\) are linearly independent over \(\mathbb{F}_q\). Then $U$ is a cyclically covering subspace of \(\mathbb{F}_{q}^n\) if and only if \(0 \notin \operatorname{supp}(\widehat{\boldsymbol{\alpha}}_j)\) for all $1\leq j\leq k$ and for every $k$-tuple \((\widehat{\boldsymbol{c}}_1,\dots,\widehat{\boldsymbol{c}}_k)\) satisfying both of the following conditions, there exists a common index \(i\in[n]\) such that  \(P_{\widehat{\boldsymbol{c}}_j}(\omega^i)=0\)~ holds simultaneously for all \(j=1,\dots,k\):

(1) For each \(j=1,\dots,k\), \(\widehat{\boldsymbol{c}}_j\in\operatorname{DFT}(\mathbb{F}_q^n)\) and \(\operatorname{supp}(\widehat{\boldsymbol{c}}_j)\subseteq -\operatorname{supp}(\widehat{\boldsymbol{\alpha}}_j)\);

(2) There exists a common \(\boldsymbol{x}\in\mathbb F_q^n\) such that for all \(j=1,\dots,k\) and all \(s\in[n]\), \((\widehat{\boldsymbol{c}}_j)_s=\widehat{\boldsymbol{x}}_s \cdot (\widehat{\boldsymbol{\alpha}}_j)_{-s}\).
\end{thm}

\begin{proof}
Since
\(U=\bigcap_{j=1}^k V_{\boldsymbol{\alpha}_j}=\big\{\,\boldsymbol{x}\in\mathbb{F}_q^n : (\boldsymbol{x},\boldsymbol{\alpha}_1)=\cdots=(\boldsymbol{x},\boldsymbol{\alpha}_k)=0\,\big\},\)
if $U$ is a cyclically covering subspace of \(\mathbb{F}_{q}^n\), then for any \(\boldsymbol{x}\in\mathbb{F}_q^n\), there exists \(i\in [n]\) such that
\(\big(\boldsymbol{x},\tau^i(\boldsymbol{\alpha}_1)\big)=\cdots=\big(\boldsymbol{x},\tau^i(\boldsymbol{\alpha}_k)\big)=0.\)
The proof of Theorem \ref{Thm2} follows from the Discrete Fourier Transform and Theorem \ref{Thm1}.
\end{proof}
When \(k=1\), Theorem \ref{Thm2} reduces to Theorem \ref{Thm1}.

\section{Applications of Theorem \ref{Thm1}}

In 2019, Aaronson, Groenland, and Johnston \cite{Aaronson-Groenland-Johnston} established that \(h_q(n) = 0\) whenever $q$ is an odd prime and \(n > q\) is a prime for which $q$ is a primitive root modulo $n$, in contrast to the case \(q=2\), where \(h_2(t) = 2\) for every prime \(t > 3\) with $2$ a primitive root modulo $t$. We now give a unified proof of this result via Theorem \ref{Thm1}.

\begin{thm}\label{Thm3.1}
Let $q$ be an odd prime and \(n > q\) a prime. If $q$ is a primitive root modulo $n$, then \(h_q(n) = 0\); by contrast, we have \(h_2(n) \geq 2\) for all primes \(n > 3\) such that $2$ is a primitive root modulo $n$.
\end{thm}

\begin{proof}
{\bf Case 1}: Let $q$ be an odd prime and \(n > q\) a prime. Suppose that $q$ is a primitive root modulo $n$. To prove \(h_q(n) = 0\), it suffices to show that no \((n-1)\)-dimensional subspace \(V_{\boldsymbol{\alpha}}\) is a cyclically covering subspace of \(\mathbb{F}_q^n\). Since \(q\) is a primitive root modulo \(n\), we have \(\gcd(q,n) = 1\) and \(\operatorname{ord}_n(q) = n-1\). Over \(\mathbb{F}_q[x]\), the polynomial \(x^n - 1\) factors as
\[x^n - 1 = (x-1)\big(1+x+\cdots+x^{n-1}\big),\]
which yields exactly two cyclotomic cosets: \(C_0 = \{0\}\) and \(C_1 = \{1,2,\dots,n-1\}\). By Corollary \ref{Lem5}, \(V_{\boldsymbol{\alpha}}\) cannot be a cyclically covering if \(0 \in \operatorname{supp}(\widehat{\boldsymbol{\alpha}})\). We therefore focus on the case where \(0 \notin \operatorname{supp}(\widehat{\boldsymbol{\alpha}})\). In this setting, \(\operatorname{supp}(\widehat{\boldsymbol{\alpha}}) = C_1\), and \(-\operatorname{supp}(\widehat{\boldsymbol{\alpha}}) = -C_1 = C_1\).

According to Lemma \ref{Lem6}, to prove that \(V_{\boldsymbol{\alpha}}\) is not a cyclically covering subspace of \(\mathbb{F}_q^n\), we only need to construct a vector \(\widehat{\boldsymbol{c}} \in \operatorname{DFT}(\mathbb{F}_q^n)\) satisfying the two conditions below:
\(\operatorname{supp}(\widehat{\boldsymbol{c}}) \subseteq C_1~(\text{i.e.,}~ \widehat{\boldsymbol{c}}_0 = 0),\) and \(P_{\widehat{\boldsymbol{c}}}(\omega^i) \neq 0~\text{for all } 0 \leq i \leq n-1,\)
where \(P_{\widehat{\boldsymbol{c}}}(t)=\operatorname{Tr}_{\mathbb{F}_q}^{\mathbb{F}_{q^{n-1}}}(\widehat{\boldsymbol{c}}_1 t)\) and \(\widehat{\boldsymbol{c}}_1\in\mathbb{F}_{q^{n-1}}\).

We consider two cases for the construction of \(\widehat{\boldsymbol{c}}_1\).

(i) When \(q\nmid n-1\), we take \(\widehat{\boldsymbol{c}}_1 = 1\), so that \(P_{\widehat{\boldsymbol{c}}}(t) = \operatorname{Tr}_{\mathbb{F}_q}^{\mathbb{F}_{q^{n-1}}}(t)\). For \(i = 0\), we have \(P_{\widehat{\boldsymbol{c}}}(\omega^0) = \operatorname{Tr}_{\mathbb{F}_q}^{\mathbb{F}_{q^{n-1}}}(1) = n-1\). Since \(q \nmid n-1\), it follows that \(n-1 \not\equiv 0 \pmod{q}\), and thus \(\operatorname{Tr}_{\mathbb{F}_q}^{\mathbb{F}_{q^{n-1}}}(1) \neq 0\). For \(1 \leq i \leq n-1\), \(\omega^i\) is also a primitive \(n\)-th root of unity, and since \(q\) is a primitive root modulo \(n\), we compute
\[\operatorname{Tr}_{\mathbb{F}_q}^{\mathbb{F}_{q^{n-1}}}(\omega^i) = \sum_{k=0}^{n-2} \omega^{i q^k} = \sum_{j=1}^{n-1} \omega^{i j} = \left( \sum_{j=0}^{n-1} \omega^{i j} \right) - 1 = 0 - 1 = -1.\]
As \(q\) is an odd prime, \(-1 \not\equiv 0 \pmod{q}\), so \(\operatorname{Tr}_{\mathbb{F}_q}^{\mathbb{F}_{q^{n-1}}}(\omega^i) \neq 0\). Therefore, \(P_{\widehat{\boldsymbol{c}}}(\omega^i) \neq 0\) for all \(i\).

(ii) When \(q \mid n-1\), we have \(\operatorname{Tr}_{\mathbb{F}_q}^{\mathbb{F}_{q^{n-1}}}(1) = n-1 \equiv 0 \pmod{q}\). In this case, we take \(\widehat{\boldsymbol{c}}_1 = 1 + \omega\), so that \(P_{\widehat{\boldsymbol{c}}}(t) = \operatorname{Tr}_{\mathbb{F}_q}^{\mathbb{F}_{q^{n-1}}}((1+\omega)t)\). For \(i = 0\), we get
\[P_{\widehat{\boldsymbol{c}}}(\omega^0) = \operatorname{Tr}_{\mathbb{F}_q}^{\mathbb{F}_{q^{n-1}}}(1+\omega) = \operatorname{Tr}_{\mathbb{F}_q}^{\mathbb{F}_{q^{n-1}}}(1) + \operatorname{Tr}_{\mathbb{F}_q}^{\mathbb{F}_{q^{n-1}}}(\omega) = 0 + (-1) = -1 \neq 0.\]
For \(i = n-1\), we have
\[P_{\widehat{\boldsymbol{c}}}(\omega^{n-1}) = \operatorname{Tr}_{\mathbb{F}_q}^{\mathbb{F}_{q^{n-1}}}((1+\omega)\omega^{n-1}) = \operatorname{Tr}_{\mathbb{F}_q}^{\mathbb{F}_{q^{n-1}}}(\omega^{n-1}) + \operatorname{Tr}_{\mathbb{F}_q}^{\mathbb{F}_{q^{n-1}}}(\omega^n) = -1 + 0 = -1 \neq 0.\]
For \(1 \leq i \leq n-2\), we obtain
\[P_{\widehat{\boldsymbol{c}}}(\omega^i) = \operatorname{Tr}_{\mathbb{F}_q}^{\mathbb{F}_{q^{n-1}}}(\omega^i) + \operatorname{Tr}_{\mathbb{F}_q}^{\mathbb{F}_{q^{n-1}}}(\omega^{i+1}) = -1 + (-1) = -2.\] Since \(q\) is an odd prime, \(-2 \not\equiv 0 \pmod{q}\). Thus, \(P_{\widehat{\boldsymbol{c}}}(\omega^i) \neq 0\) for all \(i\).

By Lemma \ref{Lem6}, we conclude that \(V_{\boldsymbol{\alpha}}\) is not a cyclically covering subspace of \(\mathbb{F}_q^n\).

{\bf Case 2}: Suppose \(n>3\) is prime and $2$ is a primitive root modulo $n$. We prove $h_2(n)\geq2$ in two steps.

{\bf Step 1}: We start with the proof that \(h_2(n) \ge 1\). Choose \(\boldsymbol{\alpha}\) such that \(\operatorname{supp}(\widehat{\boldsymbol{\alpha}}) = C_1\). For instance, we may take \(\boldsymbol{0}\neq\boldsymbol{\alpha} \in \ker\big(1+\tau+\cdots+\tau^{n-1}\big)\). Then \(0 \notin \operatorname{supp}(\widehat{\boldsymbol{\alpha}})\) and \(-\operatorname{supp}(\widehat{\boldsymbol{\alpha}}) = C_1\). We aim to prove that for any vector \(\widehat{\boldsymbol{c}} \in \operatorname{DFT}(\mathbb{F}_2^n)\) satisfying \(\operatorname{supp}(\widehat{\boldsymbol{c}}) \subseteq -\operatorname{supp}(\widehat{\boldsymbol{\alpha}})\), there exists an index \(i \in [n]\) such that \(P_{\widehat{\boldsymbol{c}}}(\omega^i) = 0\), where \(P_{\widehat{\boldsymbol{c}}}(t)=\operatorname{Tr}_{\mathbb{F}_2}^{\mathbb{F}_{2^{n-1}}}(\widehat{\boldsymbol{c}}_1 t)\) and \(\widehat{\boldsymbol{c}}_1\in\mathbb{F}_{2^{n-1}}\). We compute the sum
\[\sum_{i=0}^{n-1} P_{\widehat{\boldsymbol{c}}}(\omega^i)= \sum_{i=0}^{n-1} \operatorname{Tr}_{\mathbb{F}_2}^{\mathbb{F}_{2^{n-1}}}(\widehat{\boldsymbol{c}}_1 \omega^i)
= \operatorname{Tr}_{\mathbb{F}_2}^{\mathbb{F}_{2^{n-1}}}\left(\widehat{\boldsymbol{c}}_1 \sum_{i=0}^{n-1} \omega^i\right)= \operatorname{Tr}_{\mathbb{F}_2}^{\mathbb{F}_{2^{n-1}}}(\widehat{\boldsymbol{c}}_1 \cdot 0) = 0.\]
Thus, the number of indices $i$ for which \(P_{\widehat{\boldsymbol{c}}}(\omega^i) = 1\) is even. Since $n$ is odd, these values cannot all equal $1$. It follows that there exists at least one $i$ such that \(P_{\widehat{\boldsymbol{c}}}(\omega^i) = 0\). Thus, \(V_{\boldsymbol{\alpha}}\) is a cyclically covering subspace, so \(h_2(n) \ge 1\).

{\bf Step 2}: We now show that \(h_2(n) \ge 2\). We proceed by explicit construction. Let \(U = V_{\boldsymbol{\alpha}} \cap V_{\boldsymbol{\beta}}\), where \(\boldsymbol{\alpha}, \boldsymbol{\beta} \in \mathbb{F}_2^n\) are linearly independent over \(\mathbb{F}_2\), and \(0 \notin \operatorname{supp}(\widehat{\boldsymbol{\alpha}})\), \(0 \notin \operatorname{supp}(\widehat{\boldsymbol{\beta}})\). Since there is exactly one non-zero cyclotomic coset, it follows that \(\operatorname{supp}(\widehat{\boldsymbol{\alpha}}) = \operatorname{supp}(\widehat{\boldsymbol{\beta}}) = C_1\).

Choose any \(\widehat{\boldsymbol{\alpha}}\) such that \(\widehat{\boldsymbol{\alpha}}_1 \neq 0\). The remaining components are automatically determined by the DFT image condition \((\widehat{\boldsymbol{\alpha}}_{2^k} = \widehat{\boldsymbol{\alpha}}_1^{2^k})\), ensuring that \(\boldsymbol{\alpha} \in \mathbb{F}_2^n\). Let \(t = \frac{n-1}{2}\). Choose any
\(\lambda \in \mathbb{F}_{2^t} \setminus \{0,1\},\)
which exists because \(n>3\) implies \(t \ge 2\). Define \(\widehat{\boldsymbol{\beta}}\) with support \(C_1\) by setting
\(\widehat{\boldsymbol{\beta}}_1 = \lambda \widehat{\boldsymbol{\alpha}}_1,\)
so that \(\widehat{\boldsymbol{\beta}}_{-1} = \lambda \widehat{\boldsymbol{\alpha}}_{-1}\). Indeed, since \(-1 \equiv 2^t \pmod{n}\), we have
\[\widehat{\boldsymbol{\beta}}_{-1} = \widehat{\boldsymbol{\beta}}_{2^t} = (\widehat{\boldsymbol{\beta}}_1)^{2^t} = \left(\lambda \widehat{\boldsymbol{\alpha}}_1\right)^{2^t} = \lambda^{2^t} (\widehat{\boldsymbol{\alpha}}_1)^{2^t} = \lambda \widehat{\boldsymbol{\alpha}}_{-1}.\]
Then $\boldsymbol{\beta}\in\mathbb{F}_2^n$, and $\boldsymbol{\beta}$ is linearly independent of $\boldsymbol{\alpha}$ over $\mathbb{F}_2$ owing to $\lambda\neq1$.

Suppose $U= V_{\boldsymbol{\alpha}} \cap V_{\boldsymbol{\beta}}$ is not a cyclically covering subspace. Then there exists some \(\boldsymbol{x} \in \mathbb{F}_2^n\) such that \(\boldsymbol{x} \notin \tau^i(U)\) for every \(i \in [n]\). This implies that $( \boldsymbol{x}, \tau^i(\boldsymbol{\alpha}) ) \neq 0$ or $( \boldsymbol{x}, \tau^i(\boldsymbol{\beta}) ) \neq 0,$ for every \(i \in [n]\).
In other words, there is no index $i$ for which both inner products vanish simultaneously. Since \(\operatorname{supp}(\widehat{\boldsymbol{\alpha}}) = \operatorname{supp}(\widehat{\boldsymbol{\beta}}) = C_1\), we have
\[( \boldsymbol{x}, \tau^i(\boldsymbol{\alpha}) ) = \operatorname{Tr}_{\mathbb{F}_2}^{\mathbb{F}_{2^{n-1}}}(a \omega^{-i}), ~~ ( \boldsymbol{x}, \tau^i(\boldsymbol{\beta}) ) = \operatorname{Tr}_{\mathbb{F}_2}^{\mathbb{F}_{2^{n-1}}}(b \omega^{-i}),\]
where \(a = \widehat{\boldsymbol{x}}_1 \widehat{\boldsymbol{\alpha}}_{-1}\in\mathbb{F}_{2^{n-1}}\) and \(b = \widehat{\boldsymbol{x}}_1 \widehat{\boldsymbol{\beta}}_{-1}=\lambda a\in\mathbb{F}_{2^{n-1}}\). As $i$ runs over all indices, we may identify \(\omega^{-i}\) with \(\omega^i\), so the condition reduces to: for every $i$,
\(\left(\operatorname{Tr}_{\mathbb{F}_2}^{\mathbb{F}_{2^{n-1}}}(a \omega^i), \operatorname{Tr}_{\mathbb{F}_2}^{\mathbb{F}_{2^{n-1}}}(\lambda a \omega^i)\right) \neq (0, 0).\) We first define \(x_i = \operatorname{Tr}_{\mathbb{F}_2}^{\mathbb{F}_{2^{n-1}}}(a \omega^i)\) and \(y_i = \operatorname{Tr}_{\mathbb{F}_2}^{\mathbb{F}_{2^{n-1}}}(\lambda a \omega^i)\), both taking values in \(\{0,1\}\). It is straightforward to verify that
\[\sum_{i=0}^{n-1} x_i = \sum_{i=0}^{n-1} \operatorname{Tr}_{\mathbb{F}_2}^{\mathbb{F}_{2^{n-1}}}(a \omega^i) = 0,\]
and similarly \(\sum_{i=0}^{n-1} y_i = 0\). Furthermore, we compute
\[\begin{aligned}
\sum_{i=0}^{n-1} x_i y_i &= \sum_{i=0}^{n-1} \operatorname{Tr}_{\mathbb{F}_2}^{\mathbb{F}_{2^{n-1}}}(a\omega^i)\operatorname{Tr}(\lambda a\omega^i)\\
&= \sum_{i=0}^{n-1} \left( \sum_{k=0}^{n-2} (a\omega^i)^{2^k} \right) \left( \sum_{l=0}^{n-2} (\lambda a\omega^i)^{2^l} \right)\\
&= \sum_{k=0}^{n-2} \sum_{l=0}^{n-2} a^{2^k} (\lambda a)^{2^l} \sum_{i=0}^{n-1} \omega^{i(2^k + 2^l)}.
\end{aligned}\]
Note that
\[\sum_{i=0}^{n-1} \omega^{i(2^k + 2^l)} = 
\begin{cases}
1, & \text{if } 2^k + 2^l \equiv 0 \pmod{n},\\
0, & \text{otherwise}.
\end{cases}\]
The condition \(2^k + 2^l \equiv 0 \pmod{n}\) is equivalent to \(2^{k-l} \equiv -1 \pmod{n}\), which in turn holds if and only if \(k-l \equiv t \pmod{n-1}\), where \(t = \frac{n-1}{2}\) satisfies \(2^t \equiv -1 \pmod{n}\). When \(k = l\), we have \(2^k + 2^l = 2^{k+1} \not\equiv 0 \pmod{n}\), so these terms contribute nothing. Suppose there exists a pair \((k_0, l_0)\) with \(k_0 \neq l_0\) such that \(2^{k_0} + 2^{l_0} \equiv 0 \pmod{n}\). By symmetry, the pair \((l_0, k_0)\) also satisfies the same condition, so all such pairs occur in symmetric pairs.
The contribution of each such pair is
\[a^{2^k} (\lambda a)^{2^l} + a^{2^l} (\lambda a)^{2^k}= a^{2^k + 2^l} \left( \lambda^{2^l} + \lambda^{2^k} \right).\]
Since \(k - l \equiv t \pmod{n-1}\), we may write \(k \equiv l+t \pmod{n-1}\). Then
\(\lambda^{2^k} = (\lambda^{2^t})^{2^l} = \lambda^{2^l},\)
where the last equality holds because \(\lambda \in \mathbb{F}_{2^t}\) satisfies \(\lambda^{2^t} = \lambda\). Hence the bracketed term vanishes in characteristic 2, so the contribution of each pair is zero. All non-zero terms therefore cancel pairwise, and the total sum is zero. Thus, \(\sum_{i=0}^{n-1} x_i y_i = 0\) for any a.

Suppose for all $i$, \((x_i, y_i) \neq (0,0)\). Then the only possible non-zero pairs are \((1,0)\), \((0,1)\) and \((1,1)\). Let \(n_1, n_2, n_3\) denote the number of occurrences of these three pairs respectively, so \(n_1 + n_2+ n_3 = n\). From the vanishing sums above, we obtain
\[\sum x_i = n_1 + n_3 \equiv 0 \pmod{2}, \quad \sum y_i = n_2 + n_3 \equiv 0 \pmod{2}, \quad \sum x_i y_i = n_3 \equiv 0 \pmod{2}.\]
It follows that $n_3$ is even, and hence both $n_1$ and $n_2$ are even. This yields a contradiction, since \(n = n_1 + n_2 + n_3\) is an odd prime. Therefore, there must exist some $i$ such that \(x_i = y_i = 0\), which implies \(\boldsymbol{x} \in \tau^i(U)\). Thus $U$ is a cyclically covering subspace, and we conclude that \(h_2(n) \ge 2\).
\end{proof}

In Theorem \ref{Thm3.1}, we re-proved that if $q$ is an odd prime and \(n > q\) is a prime, then $q$ being a primitive root modulo $n$ implies \(h_q(n)=0\). However, the converse of this proposition does not necessarily hold. We present the following counterexample.

\begin{exmp}\label{3-2}
Let \(q=5\) and \(n=11\). Here, \(\operatorname{ord}_{11}(5)=5<10\), so $5$ is not a primitive root modulo $11$. We aim to prove \(h_5(11)=0\). The polynomial \(x^{11}-1\) factors as
\[x^{11}-1 = f_0(x)f_1(x)f_2(x),\]
where
\[\begin{split}
f_0(x)&=x-1,\\
f_1(x)&=x^5 + 2x^4 + 4x^3 + x^2 + x + 4,\\
f_2(x)&=x^5 + 4x^4 + 4x^3 + x^2 + 3x + 4.
\end{split}\]
Accordingly, \(\mathbb{F}_5^{11}\) decomposes into the direct sum
\[\mathbb{F}_5^{11} = W_0 \oplus W_1 \oplus W_2,\]
with \(W_i = \ker f_i(\tau)\) for \(i=0,1,2\). By Theorem \ref{Thm1} and Lemma \ref{lem:vanishing-full-weight}, it suffices to show that full-weight codewords exist in \(W_1\), \(W_2\) and \(W_1\oplus W_2\).

In fact, it is easy to verify that
\[\boldsymbol{u}_1 = (4, 2, 3, 4, 4, 3, 1, 1, 1, 1, 1),\quad
\boldsymbol{u}_2 = (3, 4, 4, 3, 2, 4, 1, 1, 1, 1, 1)\]
and \(\boldsymbol{u}_1+\boldsymbol{u}_2\) are full-weight codewords in \(W_1\), \(W_2\) and \(W_1\oplus W_2\), respectively. Hence we conclude that \(h_5(11) = 0\).

Adopting the same argument used to prove \(h_5(11)=0\), we obtain the concrete examples listed in Table \ref{tab:h5-values}. For each case where \(h_5(n)=0\), we find the corresponding full-weight codewords. We remark that among the primes listed in Table \ref{tab:h5-values} for which $5$ is not a primitive root, namely $11, 13, 19, 29, 31$, only \(h_5(31)=2>0\).
\begin{table}[htbp]
\centering 
\caption{The specific values of \(h_5(n)\) for \(6 \leq n \leq 32\)}\label{tab:h5-values} 
\vspace{0.2cm}
\begin{tabular}{ccclccclcccl}
\toprule
$n$ & $h_5(n)$ & \multicolumn{1}{c}{Reason} & \qquad & $n$ & $h_5(n)$ & \multicolumn{1}{c}{Reason} \qquad & $n$ & $h_5(n)$ & \multicolumn{1}{c}{Reason} \\
\midrule
6 & 0 & \cite{Huang} & & 15 & 0 & \cite{Aaronson-Groenland-Johnston} & 24 & 1 & \cite{Cameron-Ellis-Raynaud} \\
7 & 0 & \cite{Aaronson-Groenland-Johnston} & & 16 & 0 & Lemma \ref{lem:vanishing-full-weight} & 25 & 0 & \cite{Aaronson-Groenland-Johnston} \\
8 & 0 & Lemma \ref{lem:vanishing-full-weight} & & 17 & 0 & \cite{Aaronson-Groenland-Johnston} & 26 & 0 & \cite{Li-Yuan-Li-Zeng} \\
9 & 0 & Lemma \ref{lem:vanishing-full-weight} & & 18 & 0 & \cite{Huang} & 27 & 0 & Lemma \ref{lem:vanishing-full-weight} \\
10 & 0 & \cite{Aaronson-Groenland-Johnston} or \cite{Cameron-Ellis-Raynaud} & & 19 & 0 & Lemma \ref{lem:vanishing-full-weight} & 28 & 0 & Lemma \ref{lem:vanishing-full-weight} \\
11 & 0 & Example \ref{3-2} & & 20 & 0 & \cite{Aaronson-Groenland-Johnston} or \cite{Cameron-Ellis-Raynaud} & 29 & 0 & Lemma \ref{lem:vanishing-full-weight} \\
12 & 0 & Lemma \ref{lem:vanishing-full-weight} & & 21 & 0 & Lemma \ref{lem:vanishing-full-weight} & 30 & 0 & \cite{Li-Yuan-Li-Zeng} \\
13 & 0 & Lemma \ref{lem:vanishing-full-weight} & & 22 & 0 & \cite{Li-Yuan-Li-Zeng} & 31 & 2 & \cite{Cameron-Ellis-Raynaud} \\
14 & 0 & \cite{Huang} & & 23 & 0 & \cite{Aaronson-Groenland-Johnston} & 32 & 0 & Lemma \ref{lem:vanishing-full-weight} \\
\bottomrule
\end{tabular}
\end{table}
\end{exmp}

Although the condition that $q$ is not a primitive root modulo $n$ does not imply \(h_q(n)\ge1\) in general, we have the following result.

\begin{thm}\label{Thm3.2}
Let \(n>3\) be a prime number. If \(\operatorname{ord}_n(3)\) is odd, then \(h_3(n) \ge 1\).
\end{thm}

\begin{proof}
Let \(d=\operatorname{ord}_n(3)\). Since \(\operatorname{ord}_n(3)\) is odd, $3$ cannot be a primitive root modulo $n$. Take the non-zero cyclotomic coset \(C_1=\{1,3,3^2,\dots,3^{d-1}\}\). As \(\operatorname{ord}_n(3)\) is odd, we have \(C_1\neq -C_1\), which means \(C_1\cap (-C_1)=\emptyset\). Let \(C=-C_1\), and let \(W_C\) be the subspace of \(\mathbb{F}_3^n\) corresponding to the cyclotomic coset $C$. By Lemma \ref{lem:vanishing-full-weight}, it suffices to prove that \(W_C\) contains no full-weight codewords. Suppose, for contradiction, that there exists a full-weight codeword \(\boldsymbol{x}\in W_C\), i.e., \(x_i\in\mathbb{F}_3^\times=\{1,-1\}\) for all $i\in[n]$. For the standard inner product, we have
\[(\boldsymbol{x},\boldsymbol{x}) = \sum_{i=0}^{n-1} x_i^2 = \sum_{i=0}^{n-1} 1 = n \neq 0.\]
By Lemma \ref{Lem2}, \(\boldsymbol{x}\in W_C\) implies \(\operatorname{supp}(\widehat{\boldsymbol{x}})=C\). Combining this with Lemma \ref{Lem02} and the condition \(C\cap (-C)=\emptyset\), we get
\[(\boldsymbol{x},\boldsymbol{x}) = \frac{1}{n}\sum_{k=0}^{n-1}\widehat{\boldsymbol{x}}_k \widehat{\boldsymbol{x}}_{-k} = 0.\]
This is a contradiction. Hence, \(W_C\) has no full-weight codewords. Accordingly, by Lemma \ref{lem:vanishing-full-weight}, we obtain \(h_3(n)\geq1\).
\end{proof}

In Theorem \ref{Thm3.2}, we proved that if \(n>3\) is prime and \(\operatorname{ord}_n(3)\) is odd, then \(h_3(n)\ge 1\). We now show that there exist infinitely many such primes $n$.
\begin{lem}\label{Lem3-1}
There exist infinitely many primes \(n>3\) for which \(\operatorname{ord}_n(3)\) is odd, and the natural density of such primes among all primes is \(\dfrac{1}{3}\).
\end{lem}

\begin{proof}
It follows from Theorem 3.1.3 in \cite{Ballot} or Section 5.1 in \cite{Ballot2007} that the natural density of primes $p$ for which \(\operatorname{ord}_p(3)\) is odd is equal to \(\frac13\). While this immediately yields that there are infinitely many such primes, we present a short direct proof below.

Let \(\ell\) be an odd prime. We have
\[3^\ell-1=(3-1)\big(1+3+3^2+\cdots+3^{\ell-1}\big).\]
The sum \(1+3+3^2+\cdots+3^{\ell-1}\) is odd, so \(3^\ell-1\) has at least one odd prime divisor. Pick an odd prime divisor $n$ of \(3^\ell-1\). Then \(3^\ell \equiv 1 \pmod{n}\), which implies \(\operatorname{ord}_n(3) \mid \ell\). Since \(\ell\) is prime, \(\operatorname{ord}_n(3)\) is either $1$ or \(\ell\). If \(\operatorname{ord}_n(3)=1\), then \(3 \equiv 1 \pmod{n}\), i.e., \(n \mid 2\). This contradicts the fact that $n$ is an odd prime. Hence \(\operatorname{ord}_n(3)=\ell\). Thus, for any given odd prime \(\ell\), we can construct an odd prime $n$ such that \(\operatorname{ord}_n(3)\) is odd.

Now take two distinct odd primes \(\ell_1\) and \(\ell_2\). It holds that
\[\gcd\big(3^{\ell_1}-1,\,3^{\ell_2}-1\big) = 3^{\gcd(\ell_1,\ell_2)}-1 = 3^1-1 = 2.\]
This means the sets of odd prime divisors of \(3^{\ell_1}-1\) and \(3^{\ell_2}-1\) are disjoint. Since each \(3^\ell-1\) possesses at least one odd prime divisor, letting \(\ell\) run over all odd primes yields infinitely many distinct odd primes $n$ for which \(\operatorname{ord}_n(3)\) is odd.
\end{proof}

Contrary to Theorem \ref{Thm3.2}, for any prime \(q>3\), we establish the following result.
\begin{thm}\label{Thm3-3}
Let $q$ be a prime with \(q>3\), and let \(n>q\) also be prime. If
\[\operatorname{ord}_n(q)=\frac{n-1}{2},\quad n\not\equiv 1\pmod{q},\quad (-1)^{\frac{n-1}{2}}n\not\equiv 1\pmod{q},\]
then \(h_q(n)=0\).
\end{thm}

\begin{proof}
Since \(\operatorname{ord}_n(q)=\dfrac{n-1}{2}\), there exist exactly three cyclotomic cosets \(C_0,C_1,C_2\), where \(C_0=\{0\}\), \(C_1\) is the set of all quadratic residues, and \(C_2\) is the set of all quadratic non-residues. Consequently, \(\mathbb{F}_q^n\) admits the direct sum decomposition \(W_0\oplus W_1\oplus W_2\) with \(\operatorname{supp}(\widehat{W}_i)=C_i\) for \(i=0,1,2\). By Lemma \ref{lem:vanishing-full-weight}, to establish \(h_q(n)=0\), it suffices to verify that each of \(W_1\), \(W_2\) and \(W_1\oplus W_2\) contains at least one full-weight codeword.

We first focus on the subspace \(W_1\). Define
\[\widehat{\boldsymbol{c}}_k=
\begin{cases}
1, & k\in C_1,\\
0, & k\notin C_1.
\end{cases}\]
It is clear that \(\widehat{\boldsymbol{c}}\in \operatorname{DFT}(\mathbb{F}_q^n)\). Let
\(\boldsymbol{c}=\operatorname{IDFT}(\widehat{\boldsymbol{c}}).\)
Then
\[c_i=\frac{1}{n}\sum_{k\in C_1}\omega^{-ik}.\]
When \(i=0\), we have
\(c_0=\frac{|C_1|}{n}=\frac{d}{n}.\)
Therefore, provided that
\(q\nmid d,\)
or equivalently,
\(n\not\equiv 1\pmod{q},\)
we deduce that
\(c_0\neq 0.\)

For \(i \neq 0\), the set \(-iC_1\) is either \(C_1\) or \(C_2\). Let
\[\eta_1 = \sum_{j\in C_1}\omega^j, \qquad \eta_2 = \sum_{j\in C_2}\omega^j.\]
Hence \(c_i\) is equal to either \(\frac{\eta_1}{n}\) or \(\frac{\eta_2}{n}\). They satisfy the relation
\[\eta_1 + \eta_2 = -1.\]
We need to show that both $\eta_1$ and $\eta_2$ are nonzero, which is equivalent to $\eta_1 \eta_2 \neq 0$.

Define the quadratic Gauss sum
\[G(n) = \sum_{k=0}^{n-1} \left(\frac{k}{n}\right) \omega^k,\]
where $\left(\frac{k}{n}\right)$ denotes the Legendre symbol. It follows immediately that
\[G(n) = \eta_1 - \eta_2.\]
Since $G(n)^2 = \left(\frac{-1}{n}\right) n$, we have
\[(\eta_1 - \eta_2)^2 = \left(\frac{-1}{n}\right) n = (-1)^{\frac{n-1}{2}} n.\]
Using the identity $\eta_1 \eta_2 = \frac{(\eta_1 + \eta_2)^2 - (\eta_1 - \eta_2)^2}{4}$ and the fact that $\eta_1 + \eta_2 = -1$, we obtain
\[\eta_1 \eta_2 = \frac{1 - (-1)^{\frac{n-1}{2}} n}{4}.\]
Thus,
\[\eta_1 \eta_2 \neq 0 \quad \Longleftrightarrow \quad \frac{1 - (-1)^{\frac{n-1}{2}} n}{4} \neq 0 \quad \Longleftrightarrow \quad (-1)^{\frac{n-1}{2}} n \not\equiv 1 \pmod{q}.\]
This proves that $W_{1}$ contains a full-weight codeword whenever $n \not\equiv 1 \pmod{q}$ and $(-1)^{\frac{n-1}{2}} n \not\equiv 1 \pmod{q}$. The same construction applies verbatim to $W_{2}$.

On the other hand, $\boldsymbol c \in W_1 \oplus W_2$ if and only if $0 \notin \operatorname{supp}(\widehat{\boldsymbol c})$, i.e., vectors with coordinate sum zero modulo $q$. Since $q$ is an odd prime, $2$ is invertible modulo $q$, so the congruence $2r \equiv n \pmod{q}$ has a unique solution $r \in \{1, 2, \dots, q-1\}$. We construct a vector with $r$ coordinates equal to $1$ and the remaining $n-r$ coordinates equal to $-1$. This vector has all coordinates nonzero and sum of coordinates equal to $r - (n - r) = 2r - n \equiv 0 \pmod{q}$, hence it is a full-weight codeword in $W_{1} \oplus W_{2}$.

Therefore, it follows from Lemma \ref{lem:vanishing-full-weight} that \(h_q(n) = 0\).
\end{proof}

\begin{coro}\label{coro3-1}
Let \(n>5\) be a prime. If \(n \equiv 9 \pmod{20}\) and \(\operatorname{ord}_n(5) = \dfrac{n-1}{2}\), then \(h_5(n) = 0\).
\end{coro}

\begin{proof}
This is an immediate corollary of Theorem \ref{Thm3-3}.
\end{proof}

\begin{coro}\label{coro3-2}
Let $q$ be a prime with \(q>3\). Under the Generalized Riemann Hypothesis, there are infinitely many primes $n$ with \(n > q\) for which $q$ fails to be a primitive root modulo $n$ and \(h_q(n) = 0\).
\end{coro}

\begin{proof}
Since \(q>3\), there exists a quadratic residue \(u\in\mathbb{F}_q^\ast\) such that
\[u\not\equiv 1 \pmod q.\]
By the Chinese remainder theorem, choose a residue class \(a \pmod{4q}\) satisfying
\[a\equiv 1 \pmod 4,\qquad a\equiv u \pmod q.\]
We first observe that every prime \(n\equiv a\pmod{4q}\) automatically satisfies the two congruence conditions in Theorem \ref{Thm3-3}. Indeed,
\[n\equiv a\equiv u\not\equiv 1\pmod q,\]
so
\[n\not\equiv 1\pmod q.\]
Moreover, \(n\equiv a\equiv1\pmod4\) implies that the integer \(\frac{n-1}{2}\) is even, hence
\[(-1)^{\frac{n-1}{2}}=1,\]
and consequently
\[(-1)^{\frac{n-1}{2}}n\equiv n\equiv u\not\equiv1\pmod q.\]
It remains to verify the order condition
\[\operatorname{ord}_n(q)=\frac{n-1}{2}.\]
For a prime \(n\), this equality is equivalent to \(q\) generating the subgroup of all quadratic residues modulo \(n\), or equivalently, \(q\) being a near-primitive root of index \(2\) modulo \(n\).

The congruence class \(a\pmod{4q}\) is compatible with this order condition. Indeed, if \(n\equiv a\pmod{4q}\), then \(n\equiv1\pmod4\) and \(n\equiv u\pmod q\). By quadratic reciprocity,
\[\left(\frac{q}{n}\right)=\left(\frac{n}{q}\right)=\left(\frac{u}{q}\right)=1.\]
This shows that \(q\) is a quadratic residue modulo \(n\), so it is possible for \(q\) to generate the subgroup of quadratic residues modulo \(n\).

We now invoke the GRH-conditional Artin-type theorem for near-primitive roots in arithmetic progressions: for the fixed prime \(q>3\) and the congruence class \(a\pmod{4q}\) defined above, assuming the Generalized Riemann Hypothesis holds, there exist infinitely many primes \(n\equiv a\pmod{4q}\) such that \(q\) is a near-primitive root of index \(2\) modulo \(n\), i.e.,
\[\operatorname{ord}_n(q)=\frac{n-1}{2}.\]
This result follows from Moree's density theory for near-primitive roots \cite{Moree2013}, combined with the Artin-type density framework for arithmetic progressions established by Lenstra, Moree, and Stevenhagen \cite[Theorem 3.3, Sections 6--7]{LMS2014}.

Therefore there are infinitely many primes \(n > q\) for which \(q\) fails to be a primitive root modulo \(n\) and \(h_q(n) = 0\).
\end{proof}

\section{Algebraic Interpretation of \(h_q(mn)\ge\max\left\{h_q(m),h_q(n)\right\}\)}
In 2019, Cameron, Ellis and Raynaud \cite{Cameron-Ellis-Raynaud} proved that for any prime power $q$ and positive integers \(m,n\), the inequality
\(h_q(mn)\ge\max\big\{h_q(m),h_q(n)\big\}\)
holds. This is equivalent to \(h_q(n)\ge h_q(m)\) whenever $m$ divides $n$. The above result was established by Cameron, Ellis and Raynaud \cite{Cameron-Ellis-Raynaud} via construction. We now present an algebraic interpretation of this result.

\begin{thm}\label{thm4-1}
Let $q$ be a prime power, and let $n$ be a positive integer with $m$ a divisor of $n$. Then we have
\[h_q(n)\ge h_q(m).\]
\end{thm}

\begin{proof}
Let \(\tau_n\) and \(\tau_m\) denote the cyclic shift operators on \(\mathbb{F}_q^n\) and \(\mathbb{F}_q^m\), respectively. As \(m\mid n\), write \(n = md\). For any \(\boldsymbol{x}=(x_0,...,x_{n-1})\in\mathbb{F}_q^n\), define the map \(\pi: \mathbb{F}_q^n\to\mathbb{F}_q^m\) by
\[\pi(\boldsymbol{x})_k = \sum_{j=0}^{d-1} x_{k+jm}.\]
Take arbitrary \(\boldsymbol{x},\boldsymbol{y}\in\mathbb{F}_q^n\) and \(a,b\in\mathbb{F}_q\). We have
\[\pi(a\boldsymbol{x}+b\boldsymbol{y})_k = \sum_{j=0}^{d-1}\big(ax_{k+jm}+by_{k+jm}\big)= a\sum_{j=0}^{d-1}x_{k+jm} + b\sum_{j=0}^{d-1}y_{k+jm}= a\pi(\boldsymbol{x})_k + b\pi(\boldsymbol{y})_k.\]
This implies \(\pi(a\boldsymbol{x}+b\boldsymbol{y})=a\pi(\boldsymbol{x})+b\pi(\boldsymbol{y})\), so \(\pi\) is a linear map. For any $\boldsymbol{y}=(y_0,y_1,\dots,y_{m-1})\in\mathbb{F}_q^m$, construct $\boldsymbol{x}\in\mathbb{F}_q^n$ as follows:
$$x_i=
\begin{cases}
y_k, & i=k,\ 0\le k\le m-1,\\
0, & \text{otherwise}.
\end{cases}$$
It is easy to verify that $\pi(\boldsymbol{x})=\boldsymbol{y}$. Hence $\pi$ is surjective.

We now prove 
\[\pi \circ \tau_n = \tau_m \circ \pi,\]
which means the following diagram commutes.
\[\begin{tikzcd}
\mathbb{F}_q^{n} \arrow[r, "\pi"] \arrow[d, "\tau_{n}"']
& \mathbb{F}_q^m  \arrow[d, "\tau_m"'] \\
\mathbb{F}_q^{n} \arrow[r, "\pi"]
& \mathbb{F}_q^m 
\end{tikzcd}\]
It suffices to prove that for any \(\boldsymbol{x}\in\mathbb{F}_q^n\) and \(k\in[m]\), the equality
\[\big(\pi(\tau_n(\boldsymbol{x}))\big)_k = \big(\tau_m(\pi(\boldsymbol{x}))\big)_k\]
holds. First, we compute
\[\big(\pi(\tau_n(\boldsymbol{x}))\big)_k = \sum_{j=0}^{d-1} \big(\tau_n(\boldsymbol{x})\big)_{k+jm} = \sum_{j=0}^{d-1} x_{(k+jm-1)\bmod n}.\]
Similarly,
\[\big(\tau_m(\pi(\boldsymbol{x}))\big)_k = \pi(\boldsymbol{x})_{(k-1)\bmod m} = \sum_{j=0}^{d-1} x_{((k-1)\bmod m) + jm}.\]
When \(k \neq 0\), since
\[0 \leq k + jm - 1 \leq m - 2 + (d-1)m = n - 2,\]
it follows that
\[k + jm - 1 \pmod{n} = k + jm - 1.\]
Therefore, in this case,  
\[\{k + jm - 1 \pmod{n} : j = 0, \dots, d-1\} = \{((k-1) \bmod m) + jm : j = 0, \dots, d-1\}.\]  
When \(k = 0\), the same conclusion holds. Hence,  
\[\big(\pi(\tau_n(\boldsymbol{x}))\big)_k = \big(\tau_m(\pi(\boldsymbol{x}))\big)_k,\]  
and consequently,  
\[\pi \circ \tau_n = \tau_m \circ \pi.\]

Let $U$ be an arbitrary cyclically covering subspace of \(\mathbb{F}_q^m\). Define the set
\[V = \pi^{-1}(U) = \big\{\boldsymbol{x}\in\mathbb{F}_q^n : \pi(\boldsymbol{x})\in U\big\}.\]
We first verify that $V$ is a subspace of \(\mathbb{F}_q^n\). It follows that
\[\begin{aligned}
&\bigcup_{i=0}^{n-1}\tau_n^i(V)=\bigcup_{i=0}^{n-1}\tau_n^i\big(\pi^{-1}(U)\big)=\bigcup_{i=0}^{n-1}\pi^{-1}\big(\tau_m^i(U)\big) \\
=&~\pi^{-1}\left(\bigcup_{i=0}^{n-1}\tau_m^i(U)\right)=\pi^{-1}\left(\bigcup_{i=0}^{m-1}\tau_m^i(U)\right)=\pi^{-1}(\mathbb{F}_q^m)=\mathbb{F}_q^n.
\end{aligned}\]
Hence $V$ is a cyclically covering subspace of \(\mathbb{F}_q^n\).

Since $\pi$ is surjective, it follows from the dimension formula that
$$\dim\mathbb{F}_q^n = \dim\ker\pi + \dim\mathbb{F}_q^m,$$
which yields $\dim\ker\pi = n-m$.
Consider the restriction $\pi|_V: V\to U$ of $\pi$ to $V$. This map is surjective: for any $\boldsymbol{y}\in U$, the surjectivity of $\pi$ guarantees the existence of $\boldsymbol{x}\in\mathbb{F}_q^n$ such that $\pi(\boldsymbol{x})=\boldsymbol{y}$, and hence $\boldsymbol{x}\in V$ by definition.
Moreover, $\ker(\pi|_V) = \ker\pi$, since $\ker\pi = \pi^{-1}(\{\boldsymbol{0}\}) \subseteq \pi^{-1}(U) = V$.

Applying the dimension formula again, we have
$$\dim V = \dim\ker\pi + \dim U = n-m+\dim U.$$
We then compute the codimension
\[\operatorname{codim} V = n - \dim V = n - \big(n-m+\dim U\big) = m - \dim U = \operatorname{codim} U.\]
This implies that there exists at least one cyclically covering subspace of $\mathbb{F}_q^n$ whose codimension equals $h_q(m)$. Consequently, we obtain
$$h_q(n) \ge h_q(m).$$
\end{proof}

\section{{Tensor-Product Methods for Cyclically Covering Subspaces}}\label{sec:tensor-products}
\subsection{{The superadditivity inequality and sharpness}}\label{subsec:superadditivity}
In 2019, Aaronson, Groenland, and Johnston \cite{Aaronson-Groenland-Johnston} proved that $h_q(mn)\ge h_q(m)+h_q(n)$ holds for all positive integers $m,n$. Using Theorem \ref{Thm2} and the Discrete Fourier Transform, we provide an algebraic interpretation of $h_q(mn) \ge h_q(m) + h_q(n)$ in the coprime case $\gcd(m, n) = 1$ and $\gcd(mn, q) = 1$.
\begin{thm}\label{Thm5.1}
Let $q$ be a prime power, and let $m,n$ be positive integers with $\gcd(mn,q)=1$. If $\gcd(m,n)=1$, then
\[h_q(mn)\ge h_q(m)+h_q(n).\]
\end{thm}
The proof of Theorem \ref{Thm5.1} relies on tensor product decompositions, so we first establish some preliminary results. {Throughout the preliminary results and proof of Theorem \ref{Thm5.1}, we assume that \(\gcd(mn,q)=1\) and \(\gcd(m,n)=1\).}

Since $\gcd(m,n)=1$, the Chinese Remainder Theorem gives the following additive group isomorphism
$$\begin{aligned}
\phi:\mathbb Z_{mn}&\to \mathbb Z_m\times\mathbb Z_n\\
k&\mapsto \bigl(k\bmod{m},\ k\bmod{n}\bigr).
\end{aligned}$$
For the inverse mapping: there exist integers $s,t$ satisfying $sm+tn=1$, and
$$\phi^{-1}(a,b)=a\cdot n\cdot t+b\cdot m\cdot s\pmod{mn}.$$
Let $\operatorname{Mat}_{m\times n}(\mathbb{F}_q)$ denote the space of all $m\times n$ matrices over the finite field $\mathbb{F}_q$, and let $\mathbb{F}_q^m \otimes \mathbb{F}_q^n$ denote the tensor product of $\mathbb{F}_q^m$ and $\mathbb{F}_q^n$.

\begin{lem}\label{Lem5.1}
Let $q$ be a prime power. For any positive integers $m$ and $n$, we have
\[\mathbb{F}_q^{mn} \cong \mathbb{F}_q^m \otimes \mathbb{F}_q^n \cong \operatorname{Mat}_{m\times n}(\mathbb{F}_q).\]
\end{lem}

\begin{proof} 
We first prove that $\mathbb{F}_q^m \otimes \mathbb{F}_q^n \cong \operatorname{Mat}_{m\times n}(\mathbb{F}_q)$. Let $\{\boldsymbol{e}_i\}$ and $\{\boldsymbol{f}_j\}$ be the standard bases of $\mathbb{F}_q^m$ and $\mathbb{F}_q^n$, respectively. Then, the set
\[\{\boldsymbol{e}_i \otimes \boldsymbol{f}_j : 0 \leq i \leq m-1, 0 \leq j \leq n-1\}\]
forms a basis for $\mathbb{F}_q^m \otimes \mathbb{F}_q^n$. Let $E_{i,j} \in \operatorname{Mat}_{m\times n}(\mathbb{F}_q)$ denote the matrix with a $1$ in the $(i,j)$ position and zeros elsewhere; then, the set
\[\{E_{i,j} : 0 \leq i \leq m-1, 0 \leq j \leq n-1\}\]
forms a basis for $\operatorname{Mat}_{m\times n}(\mathbb{F}_q)$. Define a linear map 
\[\rho_1: \mathbb{F}_q^m \otimes \mathbb{F}_q^n \to \operatorname{Mat}_{m\times n}(\mathbb{F}_q),\quad \boldsymbol{e}_i \otimes \boldsymbol{f}_j \mapsto E_{i,j},\] 
and extend it linearly to the entire space. It is easy to verify that this is an isomorphism.
	
Let $\{\boldsymbol{g}_k : 0 \le k \le mn-1\}$ be a standard basis for $\mathbb{F}_q^{mn}$. Let $\rho_2 : \mathbb{F}_q^{mn} \longrightarrow \mathbb{F}_q^m \otimes \mathbb{F}_q^n$ be a linear isomorphism defined on the standard basis as
\[\rho_2(\boldsymbol{g}_k) = \boldsymbol{e}_a \otimes \boldsymbol{f}_b, \quad \text{if and only if }~~ \phi(k) = (a,b),\]
and extend it linearly to the entire space. Here, $\phi$ is the bijection $\mathbb{Z}_{mn} \to \mathbb{Z}_m \times \mathbb{Z}_n$ given by the CRT. It is easy to verify that $\rho_2$ is a linear isomorphism.
\end{proof}

\begin{rem}\label{rem5.1}
Define $\Psi : \mathbb{F}_q^{mn} \longrightarrow \operatorname{Mat}_{m \times n}(\mathbb{F}_q)$ as $\Psi = \rho_1 \circ \rho_2$. Let $X = \Psi(\boldsymbol{x})$, where $\boldsymbol{x} \in \mathbb{F}_q^{mn}$. Then
\[X_{a,b} = x_k \iff \phi(k) = (a,b).\]
The proof is given below. Since \(\Psi = \rho_1 \circ \rho_2\), for each basis vector \(\boldsymbol{g}_k\), we have
\[\Psi(\boldsymbol{g}_k) = \rho_1\big(\rho_2(\boldsymbol{g}_k)\big) = \rho_1(\boldsymbol{e}_a \otimes \boldsymbol{f}_b) = E_{a,b}.\]
Any vector \(\boldsymbol{x}\in\mathbb{F}_q^{mn}\) admits a unique linear expansion in terms of the standard basis
\[\boldsymbol{x} = \sum_{k=0}^{mn-1} x_k \boldsymbol{g}_k.\]
By linearity of \(\Psi\),
\[X=\Psi(\boldsymbol{x}) = \Psi\left(\sum_{k=0}^{mn-1} x_k \boldsymbol{g}_k\right)= \sum_{k=0}^{mn-1} x_k \cdot \Psi(\boldsymbol{g}_k)= \sum_{k=0}^{mn-1} x_k \cdot E_{a_k,b_k},\]
where \((a_k,b_k) = \phi(k)\). Consequently,
\[X_{a,b} = x_k \iff \phi(k) = (a,b).\]
\end{rem}

\begin{exmp}\label{exmp5.1}
(1) Let $\boldsymbol{u} \in \mathbb{F}_q^m$, $\boldsymbol{v} \in \mathbb{F}_q^n$. We have
\[\boldsymbol{u} = \sum_{i=0}^{m-1} u_i \boldsymbol{e}_i, \quad \boldsymbol{v} = \sum_{j=0}^{n-1} v_j \boldsymbol{f}_j.\]
Then
\[\boldsymbol{u} \otimes \boldsymbol{v} = \left(\sum_{i=0}^{m-1} u_i \boldsymbol{e}_i\right) \otimes \left(\sum_{j=0}^{n-1} v_j \boldsymbol{f}_j\right) = \sum_{i=0}^{m-1} \sum_{j=0}^{n-1} u_i v_j \cdot (\boldsymbol{e}_i \otimes \boldsymbol{f}_j).\]
Thus
\[\rho_1(\boldsymbol{u} \otimes \boldsymbol{v}) = \sum_{i=0}^{m-1} \sum_{j=0}^{n-1} u_i v_j \cdot E_{i,j} = \boldsymbol{u}\boldsymbol{v}^T.\]

(2) Let $m=2$ and $n=3$. The one-dimensional vector $\boldsymbol{x}=(x_0,x_1,x_2,x_3,x_4,x_5)$ is reshaped into the matrix
\[X= \Psi(\boldsymbol{x}) = \begin{pmatrix} 
x_0 & x_4 & x_2 \\ 
x_3 & x_1 & x_5 
\end{pmatrix}.\]
Note that this is different from the standard row-major or column-major order.
\end{exmp}

Let \(\operatorname{End}(V) = \operatorname{Hom}(V,V)\) denote the endomorphism ring of a vector space $V$. There exists a canonical ring isomorphism
\[\theta_1\colon \operatorname{End}(\mathbb{F}_q^m) \xrightarrow{\cong} \operatorname{Mat}_m(\mathbb{F}_q),\quad \mathcal{A}\mapsto A,\]
where \(\mathcal{A}\) stands for a linear operator and $A$ is its matrix representation with respect to the standard basis. Similarly, we have the canonical ring isomorphism
\[\theta_2\colon \operatorname{End}(\mathbb{F}_q^n) \xrightarrow{\cong} \operatorname{Mat}_n(\mathbb{F}_q),\quad \mathcal{B}\mapsto B.\]

Below we give the definition of the tensor product of operators.
Let \( \mathcal{A} : \mathbb{F}_q^m \to \mathbb{F}_q^m \) and \( \mathcal{B} : \mathbb{F}_q^n \to \mathbb{F}_q^n \) be two linear operators.
Their tensor product is naturally defined on \( \mathbb{F}_q^m \otimes \mathbb{F}_q^n \), specifically as follows
\[(\mathcal{A} \otimes \mathcal{B})(\boldsymbol{u} \otimes \boldsymbol{v}) = (\mathcal{A} \boldsymbol{u}) \otimes (\mathcal{B} \boldsymbol{v}),~~\boldsymbol{u}\in\mathbb{F}_q^m, \boldsymbol{v}\in \mathbb{F}_q^n,\]
and extended linearly to the whole space.

Let \(A\in\operatorname{Mat}_m(\mathbb{F}_q),\,B\in\operatorname{Mat}_n(\mathbb{F}_q)\). The matrices \(A,B\) induce a linear transformation \(A\odot B\) on \(\operatorname{Mat}_{m\times n}(\mathbb{F}_q)\) defined by
\[A\odot B\colon \operatorname{Mat}_{m\times n}(\mathbb{F}_q)\to \operatorname{Mat}_{m\times n}(\mathbb{F}_q),\quad X\mapsto AXB^{T}.\]

\begin{lem}\label{Lem5.2}
Let \(\tau_m\), \(\tau_n\) be the cyclic shift operators in \(\mathbb{F}_q^m\), \(\mathbb{F}_q^n\), respectively. Set \(T_m=\theta_1(\tau_m),\ T_n=\theta_2(\tau_n)\). The following commutative diagram commutes.
\[\begin{tikzcd}
\mathbb{F}_q^m \otimes \mathbb{F}_q^n
\arrow[r, "\rho_1", "\cong"']
\arrow[d, "\tau_m\otimes\tau_n"']
& \operatorname{Mat}_{m\times n}(\mathbb{F}_q)
\arrow[d, "T_m\odot T_n"'] \\
\mathbb{F}_q^m \otimes \mathbb{F}_q^n
\arrow[r, "\rho_1", "\cong"']
& \operatorname{Mat}_{m\times n}(\mathbb{F}_q)
\end{tikzcd}\]
\end{lem}

\begin{proof}
It suffices to verify the equality on arbitrary basis element \(\boldsymbol{e}_i\otimes \boldsymbol{f}_j\). On the one hand, we have
\[\begin{aligned}
&\rho_1\circ (\tau_m\otimes \tau_n)(\boldsymbol{e}_i\otimes \boldsymbol{f}_j)=\rho_1\big((\tau_m \boldsymbol{e}_i)\otimes(\tau_n \boldsymbol{f}_j)\big)\\
&=\rho_1\big(\boldsymbol{e}_{(i+1)\bmod m}\otimes \boldsymbol{f}_{(j+1)\bmod n}\big)=E_{(i+1)\bmod m,\,(j+1)\bmod n}.
\end{aligned}\]
On the other hand,
\[\begin{aligned}
&(T_m\odot T_n)\circ \rho_1(\boldsymbol{e}_i\otimes \boldsymbol{f}_j)=(T_m\odot T_n)(E_{i,j})=T_m E_{i,j}T_n^{T}\\
&=T_m \boldsymbol{e}_i \boldsymbol{f}_j^{T}T_n^{T}=(T_m \boldsymbol{e}_i)(T_n \boldsymbol{f}_j)^{T}=\boldsymbol{e}_{(i+1)\bmod m}\cdot\boldsymbol{f}_{(j+1)\bmod n}^{T}\\
&=E_{(i+1)\bmod m,\,(j+1)\bmod n}.
\end{aligned}\]
Since the two mappings agree on all basis vectors \(e_i\otimes f_j\), we obtain
\[\rho_1\circ (\tau_m\otimes \tau_n)=(T_m\odot T_n)\circ \rho_1.\]
\end{proof}

\begin{lem}\label{Lem5.3}
Let \(\tau_{mn}\), \(\tau_m\), \(\tau_n\) be the cyclic shift operators in \(\mathbb{F}_q^{mn}\), \(\mathbb{F}_q^m\), \(\mathbb{F}_q^n\), respectively. The following commutative diagram commutes.
\[\begin{tikzcd}
\mathbb{F}_q^{mn} \arrow[r, "\rho_2", "\cong"'] \arrow[d, "\tau_{mn}"']
& \mathbb{F}_q^m \otimes \mathbb{F}_q^n \arrow[d, "\tau_m\otimes\tau_n"'] \\
\mathbb{F}_q^{mn} \arrow[r, "\rho_2", "\cong"']
& \mathbb{F}_q^m \otimes \mathbb{F}_q^n
\end{tikzcd}\]
\end{lem}

\begin{proof} 
It suffices to verify the equality on every basis vector \(\boldsymbol{g}_k\).
Since all mappings are linear, equality on a basis implies equality over the entire ambient vector space. Fix an arbitrary integer \(k\in\{0,1,\dots,mn-1\}\), and set
\[a = k\bmod{m},\quad b = k\bmod{n},\]
so that \(\rho_2(\boldsymbol{g}_k)=\boldsymbol{e}_a\otimes \boldsymbol{f}_b\). We first evaluate the composition \(\rho_2\circ \tau_{mn}\) at \(\boldsymbol{g}_k\):
\[\begin{aligned}
\rho_2\big(\tau_{mn}(\boldsymbol{g}_k)\big) &= \rho_2\big(\boldsymbol{g}_{(k+1)\bmod mn}\big)\\
&= \boldsymbol{e}_{(k+1)\bmod m}\otimes \boldsymbol{f}_{(k+1)\bmod n}\\
&= \boldsymbol{e}_{(a+1)\bmod m}\otimes \boldsymbol{f}_{(b+1)\bmod n}.
\end{aligned}\]
Next compute \((\tau_m\otimes\tau_n)\circ\rho_2\) acting on \(\boldsymbol{g}_k\):
\[\begin{aligned}
(\tau_m\otimes\tau_n)\big(\rho_2(\boldsymbol{g}_k)\big) &= (\tau_m\otimes\tau_n)(\boldsymbol{e}_a\otimes\boldsymbol{f}_b)\\
&=\tau_m(\boldsymbol{e}_a)\otimes\tau_n(\boldsymbol{f}_b)\\
&=\boldsymbol{e}_{(a+1)\bmod{m}}\otimes \boldsymbol{f}_{(b+1)\bmod{n}}.
\end{aligned}\]
The two computations produce exactly the same result. Therefore
\(\rho_2\big(\tau_{mn}(\boldsymbol{g}_k)\big)=(\tau_m\otimes\tau_n)\big(\rho_2(\boldsymbol{g}_k)\big)\)
holds for every basis element \(\boldsymbol{g}_k\). Extending linearly to the whole space, we arrive at the identity
\[\rho_2\circ\tau_{mn}=(\tau_m\otimes\tau_n)\circ\rho_2.\]
\end{proof}

Summing up Lemma \ref{Lem5.2} and Lemma \ref{Lem5.3}, we obtain the following commutative diagram.
\[\begin{tikzcd}
\mathbb{F}_q^{mn} \arrow[r, "\rho_2", "\cong"'] \arrow[d, "\tau_{mn}"']
& \mathbb{F}_q^m\otimes\mathbb{F}_q^n \arrow[r, "\rho_1", "\cong"'] \arrow[d, "\tau_m\otimes\tau_n"']
& \operatorname{Mat}_{m\times n}(\mathbb{F}_q) \arrow[d, "T_m\odot T_n"'] \\
\mathbb{F}_q^{mn} \arrow[r, "\rho_2", "\cong"']
& \mathbb{F}_q^m\otimes\mathbb{F}_q^n \arrow[r, "\rho_1", "\cong"']
& \operatorname{Mat}_{m\times n}(\mathbb{F}_q)
\end{tikzcd}\]

{We now discuss the tensor product decomposition associated with the Discrete Fourier Transform. Set \(M=\operatorname{ord}_{mn}(q)\). Since \(\gcd(mn,q)=1\), there exists a primitive \((mn)\)-th root of unity \(\omega\in\mathbb{F}_{q^M}\) such that \(\omega^{mn}=1\) and \(\omega^i\neq1\) for every integer \(1\le i\le mn-1\). Define \(\omega_m=\omega^n\) and \(\omega_n=\omega^m\); then \(\omega_m,\omega_n\) are primitive \(m\)-th and \(n\)-th roots of unity, respectively.} Recall that for any vector \(\boldsymbol{x}\in\mathbb{F}_q^{mn}\), its Discrete Fourier Transform is the linear operator \[\operatorname{DFT}_{mn}\colon\mathbb{F}_q^{mn}\to\mathbb{F}_{q^M}^{mn}\]
defined by
\[\operatorname{DFT}_{mn}(\boldsymbol{x})_k = \sum_{j=0}^{mn-1} x_j \omega^{jk},\quad k=0,1,\dots,mn-1,\]
where \(\omega\) is a primitive $mn$-th root of unity.

\begin{lem}\label{Lem5.4}
Let \(\boldsymbol{u}\in\mathbb{F}_q^m,\boldsymbol{v}\in\mathbb{F}_q^n\) and set \(\boldsymbol{x}=\rho_2^{-1}(\boldsymbol{u}\otimes \boldsymbol{v})\in\mathbb{F}_q^{mn}\). Then for every index \(k\in[mn]\), we have
\[\widehat{\boldsymbol{x}}_k=\widehat{\boldsymbol{u}}_{kt\bmod m}\cdot \widehat{\boldsymbol{v}}_{ks\bmod n},\]
where integers \(s,t\) satisfy \(sm+tn=1\). Furthermore, the support of \(\widehat{\boldsymbol{x}}\) is given by
\[\operatorname{supp}(\widehat{\boldsymbol{x}}) = \bigl\{ an+bm \pmod{mn} : a\in\operatorname{supp}(\widehat{\boldsymbol{u}}),\; b\in\operatorname{supp}(\widehat{\boldsymbol{v}}) \bigr\}.\]
\end{lem}

\begin{proof}
Let \(\boldsymbol{x}= \rho_2^{-1}(\boldsymbol{u}\otimes \boldsymbol{v})\), \(\boldsymbol{u}\in\mathbb{F}_q^m,\boldsymbol{v}\in\mathbb{F}_q^n\). Recall that
\[\rho_2(\boldsymbol{g}_k)=\boldsymbol{e}_{k\bmod m}\otimes \boldsymbol{f}_{k\bmod n}= \boldsymbol{e}_a\otimes\boldsymbol{f}_b\]
with \(\phi(k)=(a,b)\), which implies \[\rho_2^{-1}(\boldsymbol{e}_a\otimes\boldsymbol{f}_b)=\boldsymbol{g}_{\phi^{-1}(a,b)}.\]
Here the inverse mapping is given by 
\[\phi^{-1}(a,b)=ant+bms \pmod{mn},\]
where integers \(s,t\) satisfy \(sm+tn=1\). Write the vectors as basis expansions
\[\boldsymbol{u}=\sum_{a=0}^{m-1}u_a \boldsymbol{e}_a,\quad \boldsymbol{v}=\sum_{b=0}^{n-1}v_b \boldsymbol{f}_b.\]
By bilinearity of the tensor product,
\[\boldsymbol{u}\otimes\boldsymbol{v}= \biggl(\sum_{a=0}^{m-1}u_a \boldsymbol{e}_a\biggr)\otimes\biggl(\sum_{b=0}^{n-1}v_b \boldsymbol{f}_b\biggr)=\sum_{a=0}^{m-1}\sum_{b=0}^{n-1}u_a v_b\,(\boldsymbol{e}_a\otimes\boldsymbol{f}_b).\]
Applying \(\rho_2^{-1}\) to both sides yields
\[\boldsymbol{x}= \rho_2^{-1}(\boldsymbol{u}\otimes\boldsymbol{v})=\sum_{a=0}^{m-1}\sum_{b=0}^{n-1}u_a v_b\,\rho_2^{-1}(\boldsymbol{e}_a\otimes\boldsymbol{f}_b)=\sum_{a=0}^{m-1}\sum_{b=0}^{n-1}u_a v_b \,\boldsymbol{g}_{\phi^{-1}(a,b)}.\]
Therefore for every index \(k\in[mn]\), the entry of \(\boldsymbol{x}\) reads
\[x_k = u_{k\bmod m}\cdot v_{k\bmod n}.\]
From the definition of Discrete Fourier Transform,
\[\widehat{\boldsymbol{x}}_k=\sum_{j=0}^{mn-1}x_j \omega^{jk},\quad k=0,1,\dots,mn-1,\]
and substituting the above component formula we obtain
\[\widehat{\boldsymbol{x}}_k=\sum_{j=0}^{mn-1}u_{j\bmod m}\cdot v_{j\bmod n}\,\omega^{jk}.\]
Since \(\gcd(m,n)=1\), as $j$ ranges from $0$ to \(mn-1\), the pair \((a,b)=(j\bmod m,j\bmod n)\) exhausts all elements of \(\mathbb Z_m\times\mathbb Z_n\) bijectively, with the index inversion 
\[j=\phi^{-1}(a,b)=ant+bms\pmod{mn}.\]
Changing summation variables to \(a,b\), we rewrite
\[\widehat{\boldsymbol{x}}_k=\sum_{a=0}^{m-1}\sum_{b=0}^{n-1}u_a v_b \,\omega^{\phi^{-1}(a,b)\cdot k}.\]
Set \((c,d)=\phi(k)\), so \(k=\phi^{-1}(c,d)=cnt+dms\pmod{mn}\). Compute the exponent modulo $mn$:
\[\begin{aligned}
\phi^{-1}(a,b)\cdot k &= \big(ant+bms\big)\big(cnt+dms\big)\\
&\equiv acn^2t^2+bdm^2s^2 \pmod{mn}\\
&\equiv acnt+bdms \pmod{mn}.
\end{aligned}\]
Substituting back into the DFT expression,
\[\begin{aligned}
\widehat{\boldsymbol{x}}_k
&=\sum_{a=0}^{m-1}\sum_{b=0}^{n-1}u_a v_b \,\omega^{acnt+bdms}\\
&=\sum_{a=0}^{m-1}\sum_{b=0}^{n-1}u_a v_b \,\omega_m^{act}\cdot\omega_n^{bds}\\
&=\biggl(\sum_{a=0}^{m-1}u_a \omega_m^{act}\biggr)\biggl(\sum_{b=0}^{n-1}v_b \omega_n^{bds}\biggr)\\
&=\widehat{\boldsymbol{u}}_{ct}\cdot \widehat{\boldsymbol{v}}_{ds}\\
&=\widehat{\boldsymbol{u}}_{kt\bmod m}\cdot \widehat{\boldsymbol{v}}_{ks\bmod n}.
\end{aligned}\]
From the identity \(\widehat{\boldsymbol{x}}_k=\widehat{\boldsymbol{u}}_{kt\bmod m}\cdot \widehat{\boldsymbol{v}}_{ks\bmod n}\), we have the equivalence
\[\widehat{\boldsymbol{x}}_k\neq 0 \iff \widehat{\boldsymbol{u}}_{kt\bmod m}\neq 0 \text{ and } \widehat{\boldsymbol{v}}_{ks\bmod n}\neq 0.\]
The above condition is equivalent to the existence of \(a\in \operatorname{supp}(\widehat{\boldsymbol{u}})\), \(b\in \operatorname{supp}(\widehat{\boldsymbol{v}})\) such that
\[\begin{cases}
kt\equiv a \pmod{m}\\
ks\equiv b \pmod{n}.
\end{cases}\]
Recall \(sm+tn=1\). For any fixed \(a\in\mathbb Z_m\), \(b\in\mathbb Z_n\), this system of congruences has the unique solution
\[k\equiv an+bm \pmod{mn}.\]
Consequently,
\[\operatorname{supp}(\widehat{\boldsymbol{x}})=\bigl\{an+bm\pmod{mn} : a\in\operatorname{supp}(\widehat{\boldsymbol{u}}),\,b\in\operatorname{supp}(\widehat{\boldsymbol{v}})\bigr\}.\]
\end{proof}

With all preliminary results established above, we now turn to the proof of Theorem \ref{Thm5.1}.
\begin{proof}[The proof of Theorem \ref{Thm5.1}]
Let \(k=h_q(m)\) and \(l=h_q(n)\). We prove Theorem \ref{Thm5.1} by constructing a cyclically covering subspace of \(\mathbb{F}_q^{mn}\) with codimension \(k+l\).

Since \(k=h_q(m)\) and \(l=h_q(n)\), there exist two families of vectors
\[\boldsymbol{\alpha}_1,\dots,\boldsymbol{\alpha}_k\in\mathbb{F}_q^m,\qquad \boldsymbol{\beta}_1,\dots,\boldsymbol{\beta}_l\in\mathbb{F}_q^n\]
satisfying
\[V_m=\bigcap_{i=1}^k V_{\boldsymbol{\alpha}_i},\qquad V_n=\bigcap_{j=1}^l V_{\boldsymbol{\beta}_j},\]
where \(V_m\subset\mathbb{F}_q^m\) and \(V_n\subset\mathbb{F}_q^n\) are cyclically covering subspaces of codimension $k$ and $l$, respectively.
Furthermore, by Theorem \ref{Thm2}, we may assume \(0\notin\operatorname{supp}(\widehat{\boldsymbol{\alpha}}_i)\) and \(0\notin\operatorname{supp}(\widehat{\boldsymbol{\beta}}_j)\) for all integers \(1\le i\le k\) and \(1\le j\le l\).

We define two families of new vectors in $\mathbb{F}_q^{mn}$ via the tensor product:
\[\boldsymbol{\gamma}_i = \rho_2^{-1}(\boldsymbol{\alpha}_i \otimes \boldsymbol{1}_n) \quad (1\le i\le k),\qquad
\boldsymbol{\delta}_j = \rho_2^{-1}(\boldsymbol{1}_m \otimes \boldsymbol{\beta}_j) \quad (1\le j\le l),\]
where $\boldsymbol{1}_m,\boldsymbol{1}_n$ denote the all-one vectors of $\mathbb{F}_q^m$ and $\mathbb{F}_q^n$, respectively.
For the all-one vector, its Discrete Fourier Transform satisfies $\operatorname{supp}(\widehat{\boldsymbol{1}}_n)=\{0\}$.
By Lemma \ref{Lem5.4}, we obtain
\[\operatorname{supp}(\widehat{\boldsymbol{\gamma}}_i) = \bigl\{ an+bm \pmod{mn} : a\in\operatorname{supp}(\widehat{\boldsymbol{\alpha}}_i),\ b\in\{0\} \bigr\}
= \bigl\{ an \pmod{mn} : a\in\operatorname{supp}(\widehat{\boldsymbol{\alpha}}_i) \bigr\}.\]
Similarly,
\[\operatorname{supp}(\widehat{\boldsymbol{\delta}}_j) = \bigl\{ bm \pmod{mn} : b\in\operatorname{supp}(\widehat{\boldsymbol{\beta}}_j) \bigr\}.\]
Since $0\notin\operatorname{supp}(\widehat{\boldsymbol{\alpha}}_i)$ and $\gcd(m,n)=1$, the congruence $an\equiv0\pmod{mn}$ holds if and only if $a\equiv0\pmod{m}$. This implies $0\notin\operatorname{supp}(\widehat{\boldsymbol{\gamma}}_j)$.
By symmetry, $0\notin\operatorname{supp}(\widehat{\boldsymbol{\delta}}_j)$.
Suppose $an\equiv bm\pmod{mn}$. Reducing modulo $m$, we get $an\equiv0\pmod{m}$, hence $a\equiv0\pmod{m}$, which leads to a contradiction. Therefore the two families of support sets are mutually disjoint.
The Discrete Fourier Transform is a linear isomorphism; vectors with pairwise disjoint supports are linearly independent. Consequently, the set $\{\boldsymbol{\gamma}_1,\dots,\boldsymbol{\gamma}_k,\boldsymbol{\delta}_1,\dots,\boldsymbol{\delta}_l\}$ is linearly independent.

Define
\[W = \biggl(\bigcap_{j=1}^k V_{\boldsymbol{\gamma}_j}\biggr)\bigcap \biggl(\bigcap_{j=1}^l V_{\boldsymbol{\delta}_j}\biggr).\]
It follows that $\operatorname{codim} W  = h_q(m)+h_q(n)$.
	
Assume for contradiction that W is not a cyclically covering subspace. By Theorem \ref{Thm2}, there exists a common vector \(\boldsymbol{x}\in\mathbb{F}_q^{mn}\), together with a collection of corresponding vectors
\[\boldsymbol{c}_{\gamma_1},\dots,\boldsymbol{c}_{\gamma_k},\boldsymbol{c}_{\delta_1},\dots,\boldsymbol{c}_{\delta_l}\in \mathbb{F}_{q^M}^{mn},\]
satisfying the three conditions below:

(1) Support condition:
\[\operatorname{supp}(\boldsymbol{c}_{\gamma_i}) \subseteq -\operatorname{supp}(\widehat{\boldsymbol{\gamma}}_i)\;\;(1\le i\le k),\quad \operatorname{supp}(\boldsymbol{c}_{\delta_j}) \subseteq -\operatorname{supp}(\widehat{\boldsymbol{\delta}}_j)\;\;(1\le j\le l);\]

(2) Common-vector condition:
\[(\boldsymbol{c}_{\gamma_i})_s = \widehat{\boldsymbol{x}}_s \cdot (\widehat{\boldsymbol{\gamma}}_i)_{-s}\;\;(1\le i\le k),\quad (\boldsymbol{c}_{\delta_j})_s = \widehat{\boldsymbol{x}}_s \cdot (\widehat{\boldsymbol{\delta}}_j)_{-s}\;\;(1\le j\le l);\]

(3) No-common-zero condition: There is no index \(i_0\in[mn]\) such that \(P_{\boldsymbol{c}_{\gamma_i}}(\omega^{i_0})=0\) and \(P_{\boldsymbol{c}_{\delta_j}}(\omega^{i_0})=0\) hold simultaneously for all \(1\le i\le k\), \(1\le j\le l\).

We now define vectors associated with the smaller spaces:
\[\text{For each}~\boldsymbol{c}_{\gamma_i},~\text{set}~\boldsymbol{u}_i\in \operatorname{DFT}(\mathbb{F}_q^m)~\text{via}~(\boldsymbol{u}_i)_{k_1} = (\boldsymbol{c}_{\gamma_i})_{k_1n\pmod{mn}};\]
\[\text{For each}~\boldsymbol{c}_{\delta_j},~\text{set}~\boldsymbol{v}_j\in \operatorname{DFT}(\mathbb{F}_q^n)~\text{via}~(\boldsymbol{v}_j)_{k_2} = (\boldsymbol{c}_{\delta_j})_{k_2m\pmod{mn}}.\]
It is immediate that
\[\operatorname{supp}(\boldsymbol{u}_i)\subseteq -\operatorname{supp}(\widehat{\boldsymbol{\alpha}}_i),\qquad \operatorname{supp}(\boldsymbol{v}_j)\subseteq -\operatorname{supp}(\widehat{\boldsymbol{\beta}}_j).\]

We next verify that each $\boldsymbol{u}_i$ satisfies the common-vector condition from Theorem \ref{Thm2} over $\mathbb{F}_q^m$.
From the common-vector condition,
\[(\boldsymbol{c}_{\gamma_i})_{k_1n} = \widehat{\boldsymbol{x}}_{k_1n}\cdot (\widehat{\boldsymbol{\gamma}}_i)_{-k_1n}.\]
Using the explicit DFT formula in Lemma \ref{Lem5.4} to evaluate $(\widehat{\boldsymbol{\gamma}}_i)_{-k_1n}$:
\[(\widehat{\boldsymbol{\gamma}}_i)_{-k_1n} = (\widehat{\boldsymbol{\alpha}}_i)_{(-k_1n)t \bmod m}\cdot (\widehat{\boldsymbol{1}}_n)_{(-k_1n)s \bmod n}= (\widehat{\boldsymbol{\alpha}}_i)_{-k_1 \bmod m}\cdot (\widehat{\boldsymbol{1}}_n)_{0}=n\cdot (\widehat{\boldsymbol{\alpha}}_i)_{-k_1 \bmod m}.\]
Substitution into the common-vector identity gives
\[(\boldsymbol{u}_i)_{k_1} = \widehat{\boldsymbol{x}}_{k_1n}\cdot n\cdot (\widehat{\boldsymbol{\alpha}}_i)_{-k_1\pmod{m}}.\]
Set $\widehat{\boldsymbol{y}}_{k_1} = n\cdot \widehat{\boldsymbol{x}}_{k_1n}$. Then
\[(\boldsymbol{u}_i)_{k_1} = \widehat{\boldsymbol{y}}_{k_1}\cdot (\widehat{\boldsymbol{\alpha}}_i)_{-k_1\pmod{m}},\]
which is exactly the common-vector condition of Theorem \ref{Thm2} for the space $\mathbb{F}_q^m$.
Furthermore, one can readily verify that \((\widehat{\boldsymbol{y}}_{k_1})^q = \widehat{\boldsymbol{y}}_{qk_1}\) and \(\big(({\boldsymbol{u}}_i)_{k_1}\big)^q = ({\boldsymbol{u}}_i)_{qk_1}\) for all \(k_1\in[m]\).
Therefore, by the Inverse Discrete Fourier Transform, $\widehat{\boldsymbol{y}}$ corresponds to a vector $\boldsymbol{y}\in\mathbb{F}_q^m$. By symmetry, each $\boldsymbol{v}_j$ satisfies the analogous common-vector condition of Theorem \ref{Thm2} over $\mathbb{F}_q^n$.

Now we analyze the associated trace polynomials.
For a vector \(\boldsymbol{c}\in\mathbb{F}_{q^M}^{mn}\) the trace polynomial is defined by
\[P_{\boldsymbol{c}}(t) = \sum_{s=0}^{mn-1} c_s t^s.\]
Since entries of $\boldsymbol{c}_{\gamma_i}$ vanish except at indices of the form $s=k_1n$, we have
\[P_{\boldsymbol{c}_{\gamma_i}}(t) = \sum_{k_1=0}^{m-1} (\boldsymbol{u}_i)_{k_1} t^{k_1n}= P_{\boldsymbol{u}_i}(t^{\,n}),\]
where \(P_{\boldsymbol{u}_i}(z)=\sum_{k_1} (\boldsymbol{u}_i)_{k_1} z^{k_1}\) is the trace polynomial of \(\boldsymbol{u}_i\) in \(\mathbb{F}_q^m\).
Analogously,
\[P_{\boldsymbol{c}_{\delta_j}}(t) = \sum_{k_2=0}^{n-1} (\boldsymbol{v}_j)_{k_2} t^{k_2m}= P_{\boldsymbol{v}_j}(t^{\,m}),\]
with \(P_{\boldsymbol{v}_j}(z)=\sum_{k_2} (\boldsymbol{v}_j)_{k_2} z^{k_2}\) being the trace polynomial of \(\boldsymbol{v}_j\) in $\mathbb{F}_q^n$.
	
Since $V_m$ is a cyclically covering subspace of $\mathbb{F}_q^m$, by Theorem \ref{Thm2}, there exists some $a_0\in\mathbb Z_m$ such that
\[P_{\boldsymbol{u}_i}(\omega_m^{a_0}) = 0 ~~ \text{for all }~1\le i\le k.\]
Similarly, as $V_n$ is cyclically covering in $\mathbb{F}_q^n$, there exists $b_0\in\mathbb Z_n$ satisfying
\[P_{\boldsymbol{v}_j}(\omega_n^{b_0}) = 0 \quad \text{for all }1\le j\le l.\]
By the Chinese Remainder Theorem, there is a unique integer $i_0\in[mn]$ with
\[i_0\equiv a_0\pmod{m},\qquad i_0\equiv b_0\pmod{n}.\]
We evaluate the trace polynomials at $\omega^{i_0}$.
For each $\boldsymbol{c}_{\gamma_i}$:
\[P_{\boldsymbol{c}_{\gamma_i}}(\omega^{i_0})= P_{\boldsymbol{u}_i}\big((\omega^{i_0})^{n}\big)= P_{\boldsymbol{u}_i}\big(\omega^{n i_0}\big)= P_{\boldsymbol{u}_i}\big((\omega^n)^{i_0}\big)= P_{\boldsymbol{u}_i}(\omega_m^{a_0}) = 0.\]
For each $\boldsymbol{c}_{\delta_j}$:
\[P_{\boldsymbol{c}_{\delta_j}}(\omega^{i_0})= P_{\boldsymbol{v}_j}\big((\omega^{i_0})^{m}\big)= P_{\boldsymbol{v}_j}\big(\omega^{m i_0}\big)= P_{\boldsymbol{v}_j}\big((\omega^m)^{i_0}\big)= P_{\boldsymbol{v}_j}(\omega_n^{b_0}) = 0.\]

We have thus produced a common index $i_0$ annihilating all trace polynomials simultaneously, which directly contradicts the no-common-zero condition from our initial contradictory assumption.
Hence our contrary assumption fails, and $W$ is a cyclically covering subspace of $\mathbb{F}_q^{mn}$.
We have constructed a cyclically covering subspace of $\mathbb{F}_q^{mn}$ with codimension $h_q(m)+h_q(n)$. By definition of the maximal codimension, we have
\[h_q(mn)\ge h_q(m)+h_q(n).\]
This completes the proof of Theorem \ref{Thm5.1}.
\end{proof}

\begingroup
\begin{rem}\label{rem:superadditivity-sharpness}
The superadditivity inequality can be either sharp or strict.

\(\mathrm{(i)}\) We have
\[
h_2(3)=1,\qquad h_2(5)=2,\qquad h_2(15)=3.
\]
The first two values follow from Lemmas \ref{lem1.1} \(\mathrm{(viii)}\) and \ref{lem1.2} \(\mathrm{(iv)}\), respectively. Moreover, Theorem \ref{Thm5.1} gives
\[
h_2(15)\ge h_2(3)+h_2(5)=3,
\]
whereas Lemma \ref{lem1.1} \(\mathrm{(iii)}\) gives
\[
h_2(15)\le \lfloor\log_2 15\rfloor=3.
\]
Thus \(h_2(15)=h_2(3)+h_2(5)\).

\(\mathrm{(ii)}\) Let \(t>8\) be a prime for which \(2\) is a primitive root modulo \(t\), let \(d\) be a positive integer divisible by \(t-1\), and set
\[
m=t,\qquad n=\frac{2^d-1}{t}.
\]
Since \(\operatorname{ord}_t(2)=t-1\mid d\), the integer \(t\) divides \(2^d-1\), so \(n\) is an integer. Moreover, \(n>1\) is odd; hence \(h_2(n)>0\) by Lemma \ref{lem1.1} \(\mathrm{(viii)}\). Lemma \ref{lem1.2} \(\mathrm{(iv)}\) gives \(h_2(t)=2\), while \(mn=2^d-1\) and Lemma \ref{lem1.1} \(\mathrm{(iv)}\) give \(h_2(mn)=d-1\). Finally, \(t>8\) implies
\[
n=\frac{2^d-1}{t}<2^{d-3},
\]
and therefore
\[
h_2(n)\le \lfloor\log_2 n\rfloor\le d-4.
\]
Consequently,
\[
h_2(m)+h_2(n)\le 2+(d-4)=d-2<d-1=h_2(mn).
\]
\end{rem}

\begin{prop}\label{prop:superadditivity-saturation}
Let \(m,n\) be positive integers. If
\[
h_q(m)+h_q(n)=\lfloor\log_q(mn)\rfloor,
\]
then
\[
h_q(mn)=h_q(m)+h_q(n).
\]
More generally, if \(h_q(n_i)=a_i\) for \(1\le i\le r\) and
\[
\sum_{i=1}^r a_i=\left\lfloor\log_q\left(\prod_{i=1}^r n_i\right)\right\rfloor,
\]
then
\[
h_q\left(\prod_{i=1}^r n_i\right)=\sum_{i=1}^r a_i.
\]
\end{prop}

\begin{proof}
By Lemma \ref{lem1.2} \(\mathrm{(i)}\) and Lemma
\ref{lem1.1} \(\mathrm{(iii)}\), we have
\[
h_q(m)+h_q(n)
\le h_q(mn)
\le \lfloor\log_q(mn)\rfloor
= h_q(m)+h_q(n).
\]
Hence
\[
h_q(mn)=h_q(m)+h_q(n).
\]

For the multi-factor case, repeated application of Lemma
\ref{lem1.2} \(\mathrm{(i)}\) gives
\[
\sum_{i=1}^r a_i
\le
h_q\left(\prod_{i=1}^r n_i\right).
\]
By Lemma \ref{lem1.1} \(\mathrm{(iii)}\),
\[
h_q\left(\prod_{i=1}^r n_i\right)
\le
\left\lfloor
\log_q\left(\prod_{i=1}^r n_i\right)
\right\rfloor
=
\sum_{i=1}^r a_i.
\]
Therefore,
\[
h_q\left(\prod_{i=1}^r n_i\right)
=
\sum_{i=1}^r a_i.
\]
\end{proof}

\begin{exmp}\label{ex:geometric-sum-sharpness}
Let \(q\ge3\) be a prime power, and let \(k,l,d,e\) be positive integers. Set
\[
M=1+q^k+\cdots+q^{kd},\qquad
N=1+q^l+\cdots+q^{le}.
\]
If
\[
\gcd(d+1,q^k-1)=\gcd(e+1,q^l-1)=1,
\]
then Lemma \ref{lem1.1} \(\mathrm{(vi)}\) gives
\[
h_q(M)=kd,\qquad h_q(N)=le.
\]
Furthermore,
\[
q^{kd+le}\le MN
<\frac{q^{kd+le}}{(1-q^{-k})(1-q^{-l})}
\le\frac94q^{kd+le}
<q^{kd+le+1},
\]
where the last inequality uses \(q\ge3\). Hence
\(\lfloor\log_q(MN)\rfloor=kd+le\), and Proposition
\ref{prop:superadditivity-saturation} yields
\[
h_q(MN)=h_q(M)+h_q(N)=kd+le.
\]
\end{exmp}

\subsection{{A vanishing criterion for coprime products}}\label{subsec:vanishing-products}

\begin{lem}\label{lem:crt-coset-product}
Let \(q\) be a prime power, let \(m,n\ge2\), and assume that
\(\gcd(q,mn)=\gcd(m,n)=1\). Let
\[
\phi:\mathbb Z/(mn)\mathbb Z\longrightarrow
\mathbb Z/m\mathbb Z\times\mathbb Z/n\mathbb Z
\]
be the Chinese remainder isomorphism, and write \(\phi(j)=(a,b)\).
If
\[
\gcd\bigl(\lvert C_a^{(m)}\rvert,\lvert C_b^{(n)}\rvert\bigr)=1,
\]
then
\[
\phi\bigl(C_j^{(mn)}\bigr)=C_a^{(m)}\times C_b^{(n)}.
\]
\end{lem}

\begin{proof}
The set \(\phi(C_j^{(mn)})\) is the orbit of \((a,b)\) under simultaneous multiplication by \(q\):
\[
\phi(C_j^{(mn)})
=\{(aq^r\bmod m,bq^r\bmod n):r\ge0\}.
\]
It is therefore contained in \(C_a^{(m)}\times C_b^{(n)}\), and its cardinality is
\[
\operatorname{lcm}\bigl(\lvert C_a^{(m)}\rvert,
\lvert C_b^{(n)}\rvert\bigr).
\]
Under the coprimality hypothesis, this cardinality equals
\(\lvert C_a^{(m)}\rvert\lvert C_b^{(n)}\rvert\), which is the cardinality of \(C_a^{(m)}\times C_b^{(n)}\). The inclusion is thus an equality.
\end{proof}

\begin{thm}\label{thm:coprime-product-vanishing}
Let \(q\) be a prime power and let \(m,n\ge2\). Assume that
\[
\gcd(q,mn)=1,\qquad \gcd(m,n)=1.
\]
Set
\[
d_m=\operatorname{ord}_m(q),\qquad
d_n=\operatorname{ord}_n(q).
\]
If
\[
\gcd(d_m,d_n)=1,\qquad h_q(m)=h_q(n)=0,
\]
then
\[
h_q(mn)=0.
\]
\end{thm}

\begin{proof}
For a residue \(c\) modulo an integer \(N\) coprime to \(q\), write \(W_c^{(N)}\) for the minimal cyclic code whose DFT support is \(C_c^{(N)}\). By Lemma \ref{lem:vanishing-full-weight}, it suffices to show that \(W_k^{(mn)}\) contains a full-weight codeword for every nonzero \(q\)-cyclotomic coset \(C_k^{(mn)}\).

Fix \(k\not\equiv0\pmod{mn}\) and write \(\phi(k)=(a,b)\). Since \(\phi\) is injective, \((a,b)\ne(0,0)\). The orbit lengths
\(\lvert C_a^{(m)}\rvert\) and \(\lvert C_b^{(n)}\rvert\) divide \(d_m\) and \(d_n\), respectively. Hence
\[
\gcd\bigl(\lvert C_a^{(m)}\rvert,\lvert C_b^{(n)}\rvert\bigr)=1,
\]
and Lemma \ref{lem:crt-coset-product} gives
\[
\phi(C_k^{(mn)})=C_a^{(m)}\times C_b^{(n)}.
\]

Choose integers \(s,t\) satisfying \(sm+tn=1\). Reducing this identity modulo \(m\) and modulo \(n\) shows that \(t\) is invertible modulo \(m\) and \(s\) is invertible modulo \(n\). In particular,
\[
a\ne0\Longrightarrow ta\ne0\pmod m,\qquad
b\ne0\Longrightarrow sb\ne0\pmod n.
\]
If \(a\ne0\), Lemma \ref{lem:vanishing-full-weight} applied to \(m\) provides a full-weight word
\(\boldsymbol{u}\in W_{ta}^{(m)}\); if \(a=0\), take
\(\boldsymbol{u}=\boldsymbol{1}_m\in W_0^{(m)}\). Similarly, if \(b\ne0\), choose a full-weight word
\(\boldsymbol{v}\in W_{sb}^{(n)}\), and if \(b=0\), take
\(\boldsymbol{v}=\boldsymbol{1}_n\in W_0^{(n)}\).

Set
\[
\boldsymbol{x}=\rho_2^{-1}(\boldsymbol{u}\otimes\boldsymbol{v})
\in\mathbb F_q^{mn}.
\]
The component formula in the proof of Lemma \ref{Lem5.4} gives
\[
x_r=u_{r\bmod m}v_{r\bmod n}\qquad(0\le r<mn).
\]
Every coordinate of \(\boldsymbol{u}\) and \(\boldsymbol{v}\) is nonzero, so every coordinate of \(\boldsymbol{x}\) is nonzero and
\[
\operatorname{wt}(\boldsymbol{x})=mn.
\]

Both chosen words are nonzero words in their indicated minimal cyclic codes, so Lemma \ref{Lem2} gives
\[
\operatorname{supp}(\widehat{\boldsymbol{u}})=C_{ta}^{(m)},
\qquad
\operatorname{supp}(\widehat{\boldsymbol{v}})=C_{sb}^{(n)}.
\]
By Lemma \ref{Lem5.4},
\[
\operatorname{supp}(\widehat{\boldsymbol{x}})
=\{An+Bm\bmod mn:
A\in C_{ta}^{(m)},\ B\in C_{sb}^{(n)}\}.
\]
Since \(nt\equiv1\pmod m\) and \(ms\equiv1\pmod n\), applying \(\phi\) to this support gives
\[
\phi\bigl(\operatorname{supp}(\widehat{\boldsymbol{x}})\bigr)
=C_a^{(m)}\times C_b^{(n)}
=\phi(C_k^{(mn)}).
\]
Because \(\phi\) is bijective,
\[
\operatorname{supp}(\widehat{\boldsymbol{x}})=C_k^{(mn)}.
\]
Finally, decompose \(\boldsymbol{x}\) using Lemma \ref{Lem03}. Lemma \ref{Lem2} shows that the nonzero components in that direct sum have mutually disjoint DFT supports equal to their corresponding cyclotomic cosets. Since the support of \(\widehat{\boldsymbol{x}}\) is the single coset \(C_k^{(mn)}\), all components except the one in \(W_k^{(mn)}\) vanish. Thus
\(\boldsymbol{x}\in W_k^{(mn)}\). We have constructed a full-weight word in every minimal cyclic code associated with a nonzero coset, so Lemma \ref{lem:vanishing-full-weight} yields \(h_q(mn)=0\).
\end{proof}

\begin{rem}\label{rem:order-coprimality-essential}
The multiplicative-order hypothesis in Theorem
\ref{thm:coprime-product-vanishing} cannot be omitted. Indeed,
\[
h_5(3)=0,\qquad h_5(8)=0,\qquad h_5(24)=1.
\]
Here \(h_5(3)=0\) follows from the upper bound in Lemma \ref{lem1.1} \(\mathrm{(iii)}\), and \(h_5(8)=0\) is recorded in Table \ref{tab:h5-values}; on the other hand, \(24=5^2-1\), so Lemma \ref{lem1.1} \(\mathrm{(iv)}\) gives \(h_5(24)=1\). Although \(\gcd(3,8)=1\), we have
\[
\operatorname{ord}_3(5)=\operatorname{ord}_8(5)=2,
\]
so the two orders are not coprime.
\end{rem}

\begin{exmp}\label{ex:vanishing-66}
Take \(q=5\), \(m=6\), and \(n=11\). Then
\[
\gcd(5,66)=\gcd(6,11)=1,\qquad
\operatorname{ord}_6(5)=2,\qquad
\operatorname{ord}_{11}(5)=5.
\]
Moreover, \(h_5(6)=0\) by Table \ref{tab:h5-values}, and
\(h_5(11)=0\) by Example \ref{3-2}. Since \(\gcd(2,5)=1\), Theorem
\ref{thm:coprime-product-vanishing} gives
\[
h_5(66)=0.
\]
\end{exmp}

\begin{coro}\label{cor:multifactor-vanishing}
Let \(n_1,\dots,n_r\ge2\) be pairwise coprime, let
\[
N=\prod_{i=1}^r n_i,\qquad \gcd(q,N)=1,
\]
and set \(d_i=\operatorname{ord}_{n_i}(q)\). If
\[
h_q(n_i)=0\quad(1\le i\le r),\qquad
\gcd(d_i,d_j)=1\quad(i\ne j),
\]
then \(h_q(N)=0\).
\end{coro}

\begin{proof}
Proceed by induction on \(r\). The case \(r=2\) is Theorem
\ref{thm:coprime-product-vanishing}. Suppose the result holds for \(r-1\) factors and put \(M=\prod_{i=1}^{r-1}n_i\). Then \(h_q(M)=0\), and pairwise coprimality of the \(n_i\) gives
\[
\operatorname{ord}_M(q)
=\operatorname{lcm}(d_1,\dots,d_{r-1})
=\prod_{i=1}^{r-1}d_i.
\]
This order is coprime to \(d_r\). Also \(\gcd(M,n_r)=1\), so one more application of Theorem \ref{thm:coprime-product-vanishing} yields
\(h_q(Mn_r)=h_q(N)=0\).
\end{proof}

\begin{exmp}\label{ex:vanishing-2717}
For \(q=5\), take
\[
n_1=11,\qquad n_2=13,\qquad n_3=19.
\]
The factors are pairwise coprime. Example \ref{3-2} gives \(h_5(11)=0\), while Table \ref{tab:h5-values} gives
\(h_5(13)=h_5(19)=0\). Furthermore,
\[
\operatorname{ord}_{11}(5)=5,\qquad
\operatorname{ord}_{13}(5)=4,\qquad
\operatorname{ord}_{19}(5)=9,
\]
and these three orders are pairwise coprime. Corollary
\ref{cor:multifactor-vanishing} therefore gives
\[
h_5(11\cdot13\cdot19)=h_5(2717)=0.
\]
\end{exmp}

\begin{coro}\label{cor:vanishing-product-families}
Let \(q\) be an odd prime power.

\(\mathrm{(i)}\) Let \(n,r\) be positive integers. If
\[
r\mid q-1,\qquad \gcd(rq,n)=1,\qquad h_q(n)=0,
\]
then \(h_q(rn)=0\). In particular, if \(\gcd(2q,n)=1\) and
\(h_q(n)=0\), then \(h_q(2n)=0\).

\(\mathrm{(ii)}\) Let \(n\) be a positive integer such that
\[
\gcd(q(q+1),n)=1,\qquad h_q(n)=0.
\]
Suppose in addition that either \(n=1\), or \(n\ge2\) and
\(\operatorname{ord}_n(q)\) is odd.
Then \(h_q((q+1)n)=0\).
\end{coro}

\begin{proof}
\(\mathrm{(i)}\) If \(r=1\), the conclusion is the hypothesis \(h_q(n)=0\). Since \(r\mid q-1\), the upper bound in Lemma \ref{lem1.1} \(\mathrm{(iii)}\) gives
\[
0\le h_q(r)\le\lfloor\log_q r\rfloor=0.
\]
This is also the characteristic-prime exponent \(0\) specialization of Lemma \ref{lem1.1} \(\mathrm{(vii)}\). Thus, if \(n=1\), the conclusion already follows. Now assume \(r,n\ge2\). We have \(h_q(r)=0\) as above, and
\[
\operatorname{ord}_r(q)=1.
\]
The hypotheses give \(\gcd(r,n)=\gcd(q,rn)=1\), and
\(\gcd(1,\operatorname{ord}_n(q))=1\). Theorem
\ref{thm:coprime-product-vanishing} therefore yields \(h_q(rn)=0\).

\(\mathrm{(ii)}\) If \(n=1\), Lemma \ref{lem1.3} \(\mathrm{(i)}\), with \(d=0\), gives \(h_q(q+1)=0\). Suppose \(n\ge2\). Again by Lemma \ref{lem1.3} \(\mathrm{(i)}\) with \(d=0\),
\[
h_q(q+1)=0.
\]
Since \(q\equiv-1\pmod{q+1}\) and \(q\not\equiv1\pmod{q+1}\),
\(\operatorname{ord}_{q+1}(q)=2\). This order is coprime to the odd integer \(\operatorname{ord}_n(q)\). The remaining coprimality assumptions are part of the hypotheses, so Theorem
\ref{thm:coprime-product-vanishing} gives \(h_q((q+1)n)=0\).
\end{proof}
\endgroup

\section{Galois Descent Method for \(h_{q^m}(n)\leq h_q(n)\)}
In existing literature, studies on cyclically covering subspaces over finite fields are all conducted with a fixed base field \(\mathbb{F}_q\). In this section, we employ the Galois descent method for vector spaces over finite fields to address the case of distinct base fields. Specifically, we establish the inequality \(h_{q^m}(n)\leq h_q(n)\), where $q$ is a prime power, and \(m, n\) are arbitrary positive integers. We first present some preliminaries.

Let the Galois group be \(G = \operatorname{Gal}(\mathbb{F}_{q^m}/\mathbb{F}_q) = \langle \sigma \rangle\), where the Frobenius automorphism \(\sigma\) is given by \(\sigma(x) = x^q\) and satisfies \(\sigma^m = \mathrm{id}\). For any vector \(\boldsymbol{x} = (x_0, x_1, \dots, x_{n-1}) \in \mathbb{F}_{q^m}^n\), the Galois action is defined componentwise as
\(\sigma(\boldsymbol{x}) = \big(\sigma(x_0), \sigma(x_1), \dots, \sigma(x_{n-1})\big).\) For a subspace $U$ of \(\mathbb{F}_{q^m}^n\), the orthogonal complement is defined as
\(U^\perp = \big\{\boldsymbol{x} \in \mathbb{F}_{q^m}^n : (\boldsymbol{x}, \boldsymbol{u}) = 0 \text{ for all } \boldsymbol{u} \in U\big\}.\)

\begin{lem}\label{lem5.1}
(1) For any \(\boldsymbol{x} \in \mathbb{F}_{q^m}^n\), we have \(\tau(\sigma(\boldsymbol{x})) = \sigma(\tau(\boldsymbol{x}))\).

(2) For all \(\boldsymbol{x}, \boldsymbol{y} \in \mathbb{F}_{q^m}^n\), it holds that \((\sigma(\boldsymbol{x}), \sigma(\boldsymbol{y})) = \sigma\big((\boldsymbol{x}, \boldsymbol{y})\big)\).

(3) For any subspace \(U \subseteq \mathbb{F}_{q^m}^n\) and any integer $i$, the equality \((\tau^i U)^\perp = \tau^i (U^\perp)\) holds.

(4) For any $\mathbb{F}_q$-subspace \(U_0 \subseteq \mathbb{F}_q^n\), we have
\(\big(U_0 \otimes_{\mathbb{F}_q} \mathbb{F}_{q^m}\big)^\perp = U_0^\perp \otimes_{\mathbb{F}_q} \mathbb{F}_{q^m}.\)
\end{lem}

\begin{proof}
These four properties can be verified directly; we only supply proofs for the last two. For Property (3), we have
\begin{align*}
\boldsymbol{v} \in (\tau^i U)^\perp
&\iff (\boldsymbol{v}, \tau^i \boldsymbol{u}) = 0,\ \forall \boldsymbol{u} \in U \\
&\iff (\tau^{-i} \boldsymbol{v}, \boldsymbol{u}) = 0,\ \forall \boldsymbol{u} \in U \\
&\iff \tau^{-i} \boldsymbol{v} \in U^\perp \\
&\iff \boldsymbol{v} \in \tau^i(U^\perp).
\end{align*}

Next, we establish Property (4). Let \(U_0 \subseteq \mathbb{F}_q^n\) be an \(\mathbb{F}_q\)-linear subspace. We regard \(U_0 \otimes_{\mathbb{F}_q} \mathbb{F}_{q^m}\) canonically as the \(\mathbb{F}_{q^m}\)-linear subspace of \(\mathbb{F}_{q^m}^n\) spanned by \(U_0\).
We first establish the inclusion relation. If \(\boldsymbol{v} \in U_0^\perp\), then \((\boldsymbol{u},\boldsymbol{v})=0\) holds for all \(\boldsymbol{u} \in U_0\). Accordingly, for arbitrary \(a,b \in \mathbb{F}_{q^m}\), we have \((a\boldsymbol{u},b\boldsymbol{v})=ab(\boldsymbol{u},\boldsymbol{v})=0\), which implies
\[U_0^\perp \otimes_{\mathbb{F}_q} \mathbb{F}_{q^m} \subseteq \left(U_0 \otimes_{\mathbb{F}_q} \mathbb{F}_{q^m}\right)^\perp.\]

We then compare the dimensions of both sides. Let \(\dim_{\mathbb{F}_q} U_0 = r\). It follows that \(\dim_{\mathbb{F}_{q^m}} \left(U_0 \otimes_{\mathbb{F}_q} \mathbb{F}_{q^m}\right) = r\), so \(\dim_{\mathbb{F}_{q^m}} \left(U_0 \otimes_{\mathbb{F}_q} \mathbb{F}_{q^m}\right)^\perp = n-r\). On the other hand, \(\dim_{\mathbb{F}_q} U_0^\perp = n-r\), hence \(\dim_{\mathbb{F}_{q^m}} \left(U_0^\perp \otimes_{\mathbb{F}_q} \mathbb{F}_{q^m}\right) = n-r\). Since we have derived the inclusion above and the two \(\mathbb{F}_{q^m}\)-linear subspaces share identical dimensions, then
\[\left(U_0 \otimes_{\mathbb{F}_q} \mathbb{F}_{q^m}\right)^\perp = U_0^\perp \otimes_{\mathbb{F}_q} \mathbb{F}_{q^m}.\]
\end{proof}

\begin{lem}[\bf{Galois Descent}]\cite{Bourbaki}\label{lem5.2}
Let \(U \subseteq \mathbb{F}_{q^m}^n\) be an $\mathbb{F}_{q^m}$-linear subspace. Then the following two statements are equivalent:

(1) $U$ is $G$-invariant, i.e., \(\sigma(U) = U\);

(2) There exists a unique $\mathbb{F}_q$-subspace \(U_0 \subseteq \mathbb{F}_q^n\) such that \(U = U_0 \otimes_{\mathbb{F}_q} \mathbb{F}_{q^m}\).

In this case, we have
\[\dim_{\mathbb{F}_q} U_0 = \dim_{\mathbb{F}_{q^m}} U,\]
and
\[U_0 = U^G = \big\{\boldsymbol{u} \in U : \sigma(\boldsymbol{u}) = \boldsymbol{u}\big\} = U \cap \mathbb{F}_q^n.\]
\end{lem}

\begin{proof}
This is a standard result; we refer the reader to \cite{Bourbaki}.
\end{proof}

With the aid of Galois descent, we obtain the following theorem.
\begin{thm}\label{thm5.1}
Let $q$ be a prime power, and let \(m, n\) be arbitrary positive integers. Then we have
\[h_{q^m}(n) \le h_q(n).\]
\end{thm}

\begin{proof}
Let \(k = h_{q^m}(n)\). By definition, there exists a cyclically covering subspace \(U \subseteq \mathbb{F}_{q^m}^n\) of codimension \(k\). We will construct a cyclically covering subspace of \(\mathbb{F}_q^n\) with codimension at least \(k\), which implies \(h_q(n) \ge k\).

Let \(A = U^\perp \subseteq \mathbb{F}_{q^m}^n\) be the orthogonal complement of \(U\). Then
\[\dim_{\mathbb{F}_{q^m}} A = \operatorname{codim}_{\mathbb{F}_{q^m}} U = k.\]
We define the Galois closure of \(A\) as
\[B = \sum_{\sigma \in G} \sigma(A) = \operatorname{span}_{\mathbb{F}_{q^m}} \big\{\sigma(\boldsymbol{\alpha}) \,\big|\, \boldsymbol{\alpha} \in A,\ \sigma \in G\big\}.\]
Clearly, \(B\) is a \(G\)-invariant subspace. By Galois descent (Lemma \ref{lem5.2}), there exists a unique \(\mathbb{F}_q\)-subspace \(B_0 \subseteq \mathbb{F}_q^n\) such that
\[B = B_0 \otimes_{\mathbb{F}_q} \mathbb{F}_{q^m},\]
and
\[\dim_{\mathbb{F}_q} B_0 = \dim_{\mathbb{F}_{q^m}} B \ge \dim_{\mathbb{F}_{q^m}} A = k.\]
Let \(V = B_0^\perp \subseteq \mathbb{F}_q^n\) denote the orthogonal complement of \(B_0\) in \(\mathbb{F}_q^n\). It follows that
\[\operatorname{codim}_{\mathbb{F}_q} V = \dim_{\mathbb{F}_q} B_0 \ge k.\]

To prove that \(V\) is a cyclically covering subspace of \(\mathbb{F}_q^n\), we need to verify
\[\bigcup_{i=0}^{n-1} \tau^i(V) = \mathbb{F}_q^n.\]
Equivalently, for any \(\boldsymbol{x} \in \mathbb{F}_q^n\), there exists some \(i \in [n]\) satisfying \(\boldsymbol{x} \in \tau^i(V)\).

Since \(\boldsymbol{x} \in \mathbb{F}_q^n \subset \mathbb{F}_{q^m}^n\) and \(U\) is a cyclically covering subspace of \(\mathbb{F}_{q^m}^n\), the definition yields an index \(i \in [n]\) such that \(\boldsymbol{x} \in \tau^i(U)\).
Recall that
\[\tau^i(U) = \tau^i(A^\perp) = \big(\tau^i A\big)^\perp.\]
Hence \((\boldsymbol{x},\, \tau^i \boldsymbol{\alpha}) = 0\) holds for all \(\boldsymbol{\alpha} \in A\).

For any \(\sigma \in G\), we compute:
\[\begin{aligned}
\big(\boldsymbol{x},\, \tau^i \sigma(\boldsymbol{\alpha})\big)= \big(\boldsymbol{x},\, \sigma(\tau^i \boldsymbol{\alpha})\big) 
= \big(\sigma(\boldsymbol{x}),\, \sigma(\tau^i \boldsymbol{\alpha})\big) 
= \sigma\big((\boldsymbol{x},\, \tau^i \boldsymbol{\alpha})\big) 
= \sigma(0) = 0.
\end{aligned}\]
This means that \(\boldsymbol{x}\) is orthogonal to every vector of the form \(\tau^i \sigma(\boldsymbol{\alpha})\).

As \(B = \sum_{\sigma\in G} \sigma(A)\), \(\boldsymbol{x}\) is orthogonal to all elements in \(\tau^i B\), i.e.,
\(\boldsymbol{x} \in \big(\tau^i B\big)^\perp.\)
Moreover,
\[B^\perp = \big(B_0 \otimes_{\mathbb{F}_q} \mathbb{F}_{q^m}\big)^\perp
= B_0^\perp \otimes_{\mathbb{F}_q} \mathbb{F}_{q^m}
= V \otimes_{\mathbb{F}_q} \mathbb{F}_{q^m}.\]
Combining these relations yields
\(\boldsymbol{x} \in \tau^i\big(V \otimes_{\mathbb{F}_q} \mathbb{F}_{q^m}\big).\) Note that \(\boldsymbol{x} \in \mathbb{F}_q^n\) and \(\tau^i(V) \subset \mathbb{F}_q^n\). We then obtain
\[\boldsymbol{x} \in \tau^i\big(V \otimes_{\mathbb{F}_q} \mathbb{F}_{q^m}\big) \cap \mathbb{F}_q^n = \tau^i(V).\]
We have constructed a cyclically covering subspace $V$ of \(\mathbb{F}_q^n\) whose codimension is at least $k$. Consequently,
\[h_q(n) \ge \operatorname{codim}_{\mathbb{F}_q} V \ge k = h_{q^m}(n),\]
which implies the inequality 
\[h_{q^m}(n) \le h_q(n)\]
holds for all cases. This completes the proof of Theorem \ref{thm5.1}.
\end{proof}

\section{Algebraic Interpretations of Other Examples}
In 2019, Cameron, Ellis and Raynaud \cite{Cameron-Ellis-Raynaud} proved the following two results (see Lemma \ref{lem1.1}):
\[h_q(q^d - 1) = d - 1 = \lfloor\log_q(q^d - 1)\rfloor,\]
\[h_q\left(\sum_{r=0}^d q^{kr}\right) = kd,~\text{where} ~\gcd(d + 1, q^k - 1) = 1.\]
In this section, we provide a unified algebraic interpretation of these two results.

\begin{thm}\label{Thm7-1}
Let $q$ be a prime power, $m$ a positive integer, and $h$ a positive divisor of \(q-1\). If \(\gcd\left(m, h\right) = 1\), then
\[h_q\left(\frac{q^m - 1}{h}\right) = m-1=\left\lfloor \log_q\left( \frac{q^m - 1}{h} \right) \right\rfloor.\]
\end{thm}

\begin{proof}
Let $q$ be a prime power, $m$ a positive integer, and $h \mid q-1$ such that $\gcd\left(\dfrac{q^m-1}{q-1},\,h\right)=\gcd\left(m, h\right)=1$. Denote
$$n = \frac{q^m-1}{h}.$$

The multiplicative group $\mathbb{F}_{q^m}^\ast$ of the finite field $\mathbb{F}_{q^m}$ is a cyclic group of order $q^m-1$. Let $\gamma$ be a generator of $\mathbb{F}_{q^m}^\ast$.
Define the subgroup $H = \langle \gamma^n \rangle$. From $n = \dfrac{q^m-1}{h}$, we have $|H|=h$. Since $h \mid q-1$, it follows that $H \subseteq \mathbb{F}_q^\ast$.
Construct the quotient group $G = \mathbb{F}_{q^m}^\ast / H$, which is also cyclic of order $n$, generated by $\bar{\gamma} = \gamma H$.

Take $\beta = \gamma^h$. The order of $\beta$ is
$$\frac{q^m-1}{\gcd(q^m-1,\,h)} = n,$$
so $\beta^n=1$ and $\bar{\beta} = \beta H = \bar{\gamma}^h$.

Let
$$\Phi: \mathbb{F}_q^n \to \mathbb{F}_{q^m},\quad \Phi(\boldsymbol{v}) = \sum_{i=0}^{n-1} v_i\, \beta^i.$$

If $m=1$, then $\mathbb{F}_{q^m}=\mathbb{F}_q$, and $\Phi$ is clearly surjective as its image is the entire field $\mathbb{F}_q$.

If $m\ge 2$, we have
$$n = \frac{q^m-1}{h} \geq \frac{q^m-1}{q-1} \ge q^{m-1}+1.$$
This implies the order of $\beta$ is greater than $q^{m-1}$, so $\beta$ does not lie in any proper subfield of $\mathbb{F}_{q^m}$. Consequently, the minimal polynomial of $\beta$ over $\mathbb{F}_q$ has degree $m$, and $\{1,\beta,\beta^2,\dots,\beta^{m-1}\}$ forms an $\mathbb{F}_q$-basis for $\mathbb{F}_{q^m}$. The set $\{\beta^i \mid 0\le i\le n-1\}$ therefore spans the entire extension field, which means $\Phi$ is surjective.

By the dimension formula for linear maps, we get
$$\dim \ker \Phi = n - m.$$

Take an arbitrary $1$-dimensional $\mathbb{F}_q$-subspace $W$ of $\mathbb{F}_{q^m}$ (a typical choice is $W=\mathbb{F}_q$). Define
$$U = \big\{\, \boldsymbol{v}\in\mathbb{F}_q^n : \Phi(\boldsymbol{v})\in W \,\big\}.$$
By properties of linear maps, $U$ is a linear subspace. The restriction of $\Phi$ to $U$ has kernel $\ker\Phi$ and image $W$, so
$$\dim U = \dim\ker\Phi + 1 = n - m + 1.$$
The codimension of $U$ is then
$$\operatorname{codim}U = m-1.$$

For any vector $\boldsymbol{v}$, compute the image of its shift under $\Phi$:
$$\begin{aligned}
\Phi(\tau \boldsymbol{v})
&= \sum_{i=0}^{n-1} (\tau \boldsymbol{v})_i\, \beta^i= v_{n-1} + \sum_{i=1}^{n-1} v_{i-1}\, \beta^i \\
&= v_{n-1} + \beta\sum_{j=0}^{n-2} v_j\, \beta^j \\
&= \beta\sum_{j=0}^{n-1} v_j\, \beta^j + v_{n-1}\big(1-\beta^n\big).
\end{aligned}$$
Since $\beta^n=1$, the last term vanishes. Hence for all $\boldsymbol{v}\in\mathbb{F}_q^n$,
$$\Phi(\tau \boldsymbol{v}) = \beta\,\Phi(\boldsymbol{v}).$$
Define the map $T\colon \mathbb{F}_{q^m}\to\mathbb{F}_{q^m}$ by $T(x) = \beta x$. Then we have the following commutative diagram:
\[\begin{tikzcd}
\mathbb{F}_q^{n} \arrow[r, "\Phi"] \arrow[d, "\tau"']
& \mathbb{F}_{q^m}  \arrow[d, "T"'] \\
\mathbb{F}_q^{n} \arrow[r, "\Phi"]
& \mathbb{F}_{q^m} 
\end{tikzcd}\]

We now verify that for any nonzero vector $\boldsymbol{v}\in \mathbb{F}_q^n$, there exists an integer $i\in[n]$ such that $\tau^i \boldsymbol{v}\in U$. By definition,
\[\tau^i \boldsymbol{v}\in U \iff \Phi(\tau^i \boldsymbol{v})\in W.\]
Recall that $\Phi(\tau^i \boldsymbol{v})=\beta^i\Phi(\boldsymbol{v})$, which further implies
\[\Phi(\tau^i \boldsymbol{v})\in W \iff \beta^i \Phi(\boldsymbol{v})\in W.\]
Let $\boldsymbol{v}\in \mathbb{F}_q^n$ be a nonzero vector. If $\Phi(\boldsymbol{v})=0$, then $\boldsymbol{v}\in\ker\Phi\subseteq U$, and we may take $i=0$.
If $\Phi(\boldsymbol{v})\neq 0$, it suffices to find an integer $i$ such that $\tau^i \boldsymbol{v}\in U$, which is equivalent to
$$\beta^i \Phi(\boldsymbol{v}) \in W = \mathbb{F}_q.$$

Introduce the projective space
$$\mathbb{P} = \big\{\, [y] = \mathbb{F}_q^\ast\cdot y : y\in\mathbb{F}_{q^m}^\ast \big\},$$
with cardinality $|\mathbb{P}| = \dfrac{q^m-1}{q-1}$.
The group $G=\mathbb{F}_{q^m}^\ast/H$ acts on $\mathbb{P}$ by $(gH)\cdot [y] = [gy]$. The stabilizer of $[1]$ is
$$S = \operatorname{Stab}_G([1]) = \mathbb{F}_q^\ast / H,\quad |S| = \frac{q-1}{h}.$$
We have
$$\frac{|G|}{|S|} = \dfrac{n}{\frac{q-1}{h}} = \frac{q^m-1}{q-1} = |\mathbb{P}|,$$
so $G$ acts transitively on $\mathbb{P}$.

Let
$$A = \frac{q^m-1}{q-1},\qquad B = \frac{q-1}{h},$$
so $n=AB$, and the given condition is $\gcd(A,h)=1$.

Let $K=\langle \bar{\beta} \rangle$ be the subgroup generated by $\bar{\beta}=\bar{\gamma}^h$. The order of $K$ is $\dfrac{n}{\gcd(n,h)}$.
Using $\gcd(A,h)=1$, we obtain
$$\gcd(n,h) = \gcd(AB,h) = \gcd(B,h).$$
Set $d=\gcd(B,h)$, then $|K|=\dfrac{n}{d}$.

Direct computation in the cyclic group $G$ gives $|K\cap S|=\dfrac{B}{d}$.
By the formula for the order of a product of groups:
$$
|KS| = \frac{|K|\cdot |S|}{|K\cap S|}
= \frac{\frac{n}{d}\cdot B}{\frac{B}{d}}
= n,
$$
which implies $KS=G$. Thus $K$ acts transitively on the coset space $G/S \cong \mathbb{P}$.

By transitivity, there exists an integer $i$ such that $\bar{\beta}^i\cdot [1] = [\Phi(\boldsymbol{v})]$, i.e.
$$\beta^{-i} \Phi(\boldsymbol{v}) \in \mathbb{F}_q^\ast.$$
This yields $\tau^{-i} \boldsymbol{v}\in U$. Therefore $U$ is a cyclically covering subspace of $\mathbb{F}_q^n$.

In summary, $U$ is a cyclically covering subspace, which yields the lower bound $h_q(n)\ge m-1$. Combining this with the upper bound
\[h_q(n)\le\left\lfloor \log_q\left( \frac{q^m - 1}{h} \right) \right\rfloor = m-1,\]
we eventually conclude that $h_q(n)=m-1$.
\end{proof}

\begin{coro}\label{coro7-1}
(1) When \(H=\{1\}\), we have \(h=1\). The condition
\(\gcd\left(m, h\right) = 1\)
always holds. Consequently, we obtain
\[h_q(q^m - 1) = m - 1 = \lfloor\log_q(q^m - 1)\rfloor.\]
This is precisely Theorem 5 established by Cameron, Ellis and Raynaud \cite{Cameron-Ellis-Raynaud} (see also part (iv) of Lemma \ref{lem1.1}).

(2) When \(H=\mathbb{F}_q^*\), we have \(h=q-1\). Accordingly, we get
\[h_q\left(\frac{q^m - 1}{q-1}\right) = h_q\left(\sum_{r=0}^{m-1} q^r\right) = m - 1,~~\text{where}~\gcd(m,q-1) = 1.\]
This is exactly the case when $k=1$ in Theorem 8 proved by Cameron, Ellis and Raynaud \cite{Cameron-Ellis-Raynaud} (see also part (vi) of Lemma \ref{lem1.1}).

The following consequences include new families beyond the two previously known cases.

(3) Let $q$ be an odd prime power. Let $H$ be the subgroup of the multiplicative group \(\mathbb{F}_q^*\) of order \(\dfrac{q-1}{2}\). Then \(h=\dfrac{q-1}{2}\). If $\gcd\left(m,\frac{q-1}{2}\right)=1,$ then
\[h_q\left(\frac{q^m-1}{(q-1)/2}\right)=h_q\left(2\,\frac{q^m-1}{q-1}\right)=m-1.\]
When \(m=2\), the condition simplifies to \(q\equiv 3\pmod{4}\), and we have
\[h_q\big(2(q+1)\big)=1.\]

(4) Let $m$ be odd. Let $q$ be an odd prime power such that $q-1=2^k\ (k\ge 1)$. Then for any $h\mid q-1$, we have
\[\gcd\left(\frac{q^m-1}{q-1},\,h\right)=1.\]
Consequently, for all $h\mid q-1$, we have
\[h_q\left(\frac{q^m-1}{h}\right)=m-1.\]

(5) Let $q=2^k$ with $k\ge 1$. Then $q-1=2^k-1$ is odd. For any divisor $h$ of $2^k-1$, we have
\[\gcd\left(\frac{q^m-1}{q-1},\,h\right)=\gcd(m,h).\]
If $\gcd(m,h)=1$, then
\[h_{2^k}\left(\frac{2^{km}-1}{h}\right)=m-1.\]
\end{coro}

\section{Average Lower Bounds for \(h_q(n)\)}
Let $n$ be a prime with $2$ as a primitive root. By the results in \cite{Aaronson-Groenland-Johnston} (see Lemma \ref{lem1.2}), we have \(h_2(n)=2\). Since \(h_2(mn)\ge h_2(m)+h_2(n)\), there exists a sequence \(\{n_i\}\) such that \(h_2(n_i)\to\infty\) as \(n_i\to\infty\). Furthermore, in Theorem \ref{Thm3.2} and Lemma \ref{Lem3-1}, we show that for any prime \(n>3\) with \(\operatorname{ord}_n(3)\) odd, the inequality \(h_3(n)\ge 1\) holds. Such primes make up a proportion of \(\frac13\) among all primes. Combining this with the inequality \(h_3(mn)\ge h_3(m)+h_3(n)\), we similarly obtain a sequence \(\{n_i\}\) satisfying \(h_3(n_i)\to\infty\) as \(n_i\to\infty\). Accordingly, we investigate average lower bounds for \(h_q(n)\), and divide our discussion into two cases: \(q=2\) and \(q=3\).

We need two classical results concerning prime numbers: the Prime Number Theorem and Mertens' second theorem~(see \cite{Apostol}). Let \(\pi(x)\) denote the number of primes not exceeding $x$. The Prime Number Theorem asserts that
\[\pi(x)\sim \frac{x}{\log x}\quad (x\to\infty).\]
On the other hand, Mertens' second theorem gives the asymptotic formula for the sum of reciprocals of all primes not exceeding $x$:
\[\sum_{p \le X} \frac{1}{p} = \log\log X + M + o(1),\]
where $M$ is the Meissel-Mertens constant.

We also require Abel's summation formula.
\begin{lem}\cite{Apostol}\label{Lem4.1}
Let \(\{a_n\}_{n=1}^{\infty}\) be a sequence and let \(f(t)\) be a continuously differentiable function. Define the partial sum
\(A(t) = \sum_{1 \le n \le t} a_n.\)
Then for any real number \(x \ge 1\), the identity
\[\sum_{1 \le n \le x} a_n f(n) = A(x) f(x) - \int_1^x A(t) f'(t) \, dt\]
holds.
\end{lem}

We now state the following theorem.
\begin{thm}\label{Thm4.1}
Let $q$ be a prime number. Assume that there exists a set of primes $\mathcal{P}$ of natural density $A_q > 0$ and a constant $c_q > 0$ such that $h_q(p) \ge c_q$ for every $p \in \mathcal{P}$.  
Then we obtain the following average lower bound
\[\frac{1}{X}\sum_{n\le X} h_q(n) \ge (c_q A_q + o(1))\log\log X \qquad (X\to\infty).\]
\end{thm}

\begin{proof}
By the results in \cite{Aaronson-Groenland-Johnston} (see Lemma \ref{lem1.2}), we have
\[h_q(mn) \ge h_q(m) + h_q(n) \qquad\text{for all positive integers } m,n.\]
In particular, if $\operatorname{rad}(n) = \prod_{p\mid n}p$ denotes the radical of $n$, then
\[h_q(n) \ge h_q(\operatorname{rad}(n)) \ge \sum_{p\mid n} h_q(p).\]
Let $w_{\mathcal{P}}(n)$ be the number of distinct prime factors of $n$ that belong to $\mathcal{P}$, i.e., $w_{\mathcal{P}}(n) = \displaystyle\sum_{\substack{p \in \mathcal{P} \\ p\mid n}} 1$. Since $h_q(p) \ge c_q$ for all $p\in \mathcal{P}$, we immediately get
\[h_q(n) \ge c_q \, w_{\mathcal{P}}(n).\]

Summing over all integers $n\le X$ gives
\[\sum_{n\le X} h_q(n) \ge c_q \sum_{n\le X} w_{\mathcal{P}}(n).\]
Interchanging the order of summation yields
\[\sum_{n\le X} w_{\mathcal{P}}(n)= \sum_{\substack{p\in \mathcal{P}\\ p\le X}} \sum_{\substack{n\le X\\ p\mid n}} 1
=\sum_{\substack{p\in \mathcal{P}\\ p\le X}} \Bigl\lfloor \frac{X}{p} \Bigr\rfloor.\]
Writing
\[\left\lfloor \frac{X}{p} \right\rfloor = \frac{X}{p} - \left\{\frac{X}{p}\right\}\]
with $0\le \left\{\frac{X}{p}\right\}<1$, we obtain
\begin{equation}\label{eq:main-ineq}
\sum_{n\le X} h_q(n) \ge c_q \Bigl( X \sum_{\substack{p\in \mathcal{P}\\ p\le X}} \frac{1}{p}-\sum_{\substack{p\in \mathcal{P}\\ p\le X}} \Bigl\{ \frac{X}{p} \Bigr\} \Bigr).
\end{equation}	
Next, we evaluate the sum $\displaystyle\sum_{\substack{p \in \mathcal{P} \\ p \le X}} \frac{1}{p}$ by means of Abel's summation formula. Define
\[a_p =
\begin{cases}
1, & \text{if } p \in \mathcal{P} \text{ and } p \text{ is prime}, \\
0, & \text{otherwise},
\end{cases}\]
and let $A(t) = \displaystyle\sum_{p \le t} a_p = \pi_{\mathcal{P}}(t)$, which counts the number of primes in $\mathcal{P}$ up to $t$. Set $f(t) = \dfrac{1}{t}$. By Abel's summation formula,
\[\sum_{p \le X} a_p f(p) = A(X)f(X) - \int_{2}^{X} A(t) f'(t) dt.\]
Substituting the corresponding expressions, we get
\[\sum_{\substack{p\in \mathcal{P}\\ p\le X}} \frac{1}{p}= \frac{\pi_P(X)}{X} + \int_{2}^{X} \frac{\pi_P(t)}{t^2}\,dt.\]
Substituting $\pi_P(t) = A_q \pi(t) + R(t)$ with $R(t)=o\left(\frac{t}{\log t}\right)$, we obtain
\[\sum_{\substack{p\in \mathcal{P}\\ p\le X}} \frac{1}{p}= A_q\Bigl( \frac{\pi(X)}{X} + \int_{2}^{X} \frac{\pi(t)}{t^2}\,dt \Bigr)+ \frac{R(X)}{X} + \int_{2}^{X} \frac{R(t)}{t^2}\,dt.\]
Recall Mertens' second theorem:
\[\sum_{p \le X} \frac{1}{p} = \log\log X + M + o(1),\]
where $M$ is the Meissel-Mertens constant. Applying Abel's summation formula once again (with $a_n=1$ for prime $n$ and $a_n=0$ otherwise, and $f(t)=\frac{1}{t}$) verifies that
\[\frac{\pi(X)}{X} + \int_{2}^{X} \frac{\pi(t)}{t^2} dt = \sum_{p \le X} \frac{1}{p}.\]
Moreover, $\frac{R(X)}{X} = o\left(\frac{1}{\log X}\right) \to 0$ and, because $R(t)=o\left(\frac{t}{\log t}\right)$,
\[\int_{2}^{X} \frac{R(t)}{t^2}\,dt = \int_{2}^{X} o\!\left(\frac{1}{t\log t}\right)dt = o(\log\log X).\]
Consequently,
\begin{equation}\label{eq:recip-sum}
\sum_{\substack{p\in \mathcal{P}\\ p\le X}} \frac{1}{p} = A_q \log\log X + o(\log\log X).
\end{equation}

For the fractional part sum we use the trivial estimate
\begin{equation}\label{eq:frac-sum}
\sum_{\substack{p\in \mathcal{P}\\ p\le X}} \Bigl\{ \frac{X}{p} \Bigr\}\le \sum_{\substack{p\in \mathcal{P}\\ p\le X}} 1 = \pi_{\mathcal{P}}(X) \sim A_q \frac{X}{\log X}= o(X\log\log X).
\end{equation}	
Inserting \eqref{eq:recip-sum} and \eqref{eq:frac-sum} into \eqref{eq:main-ineq} gives
\[\sum_{n\le X} h_q(n)\ge c_q \Bigl( X\bigl(A_q \log\log X + o(\log\log X)\bigr) - o(X\log\log X) \Bigr)= c_q A_q X \log\log X + o(X\log\log X).\]
Dividing by $X$ we finally obtain
\[\frac{1}{X}\sum_{n\le X} h_q(n) \ge (c_q A_q + o(1))\log\log X.\]
\end{proof}

Artin conjectured that $2$ is a primitive root modulo infinitely many primes. More generally, for any positive integer $q$ that is not a perfect square, there are infinitely many primes $n$ for which $q$ is a primitive root modulo $n$. This conjecture holds under the Generalized Riemann Hypothesis (Hooley \cite{Hooley}). Moreover, these primes have a positive natural density, denoted by \(A(q)\). For example,
\[A(2) = \prod_{p \text{ prime}} \left( 1 - \frac{1}{p(p-1)} \right) \approx 0.3739558136 \dots.\]

From Artin's primitive root conjecture, it follows that
\begin{coro}\label{coro4.1}
Under the Generalized Riemann Hypothesis, we have
\[\frac{1}{X}\sum_{n\leq X}h_2(n)\geq (2A(2)+o(1))\log\log X~~~(X \to +\infty).\]
\end{coro}

\begin{proof}
We take the set of primes $\mathcal{P}$ in Theorem \ref{Thm4.1} as
\[\mathcal{P}=\{p \text{ prime} : p>3 ~\text{ and }~ 2~ \text{ is a primitive root modulo } p\}.\]
In this case, \(c_2=2\). Consequently, Corollary \ref{coro4.1} follows from Theorem \ref{Thm4.1}.
\end{proof}

Unlike the case of $q=2$, we do not need to assume the Generalized Riemann Hypothesis for $q=3$. We now state the following corollary.
\begin{coro}\label{coro4.2}
Let $q=3$. We have
$$\frac{1}{X}\sum_{n\le X}h_3(n) \ge \left(\frac13+ o(1)\right)\log\log X \quad (X\to +\infty).$$
\end{coro}

\begin{proof}
In Theorem \ref{Thm3.2} and Lemma \ref{Lem3-1}, we show that for any prime \(n>3\) with \(\operatorname{ord}_n(3)\) odd, the inequality \(h_3(n)\ge 1\) holds. Such primes make up a proportion of \(\frac13\) among all primes. Define the set of primes
$$\mathcal{P}=\{p \text{ prime} : p>3~ \text{ and } ~\operatorname{ord}_p(3)~\text{is odd}\}.$$
Consequently, Corollary \ref{coro4.2} follows from Theorem \ref{Thm4.1}.
\end{proof}

\end{document}